\documentclass{amsart} 
\usepackage{amsmath}
\usepackage[mathscr]{eucal} 
\usepackage{amssymb}
\usepackage{latexsym}
\usepackage{amsthm} 
\theoremstyle{plain}
\newtheorem{theorem}{Theorem}[section]

\newtheorem{proposition}[theorem]{Proposition}

\newtheorem{corollary}[theorem]{Corollary}

\newtheorem{lemma}[theorem]{Lemma}  
\newcommand{\supp}{\mathop{\mathrm{supp}}\nolimits} 
 
\newcommand{\card}{\mathop{\mathrm{card}}\nolimits}

\numberwithin{equation}{section}  
\theoremstyle{definition}

\theoremstyle{remark}

\def\XXint#1#2#3{{\setbox0=\hbox{$#1{#2#3}{\int}$}
\vcenter{\hbox{$#2#3$}}\kern-.5\wd0}}

\title[Estimates for maximal Fourier multiplier operators]
{Estimates for maximal Fourier multiplier operators   
on $\Bbb R^2$ via square functions}  
\author{Shuichi Sato } 

\begin{document} 
\address{
Kanazawa Japan}
\email{shuichipm@gmail.com}

\begin{abstract} 
We consider certain Littlewood-Paley square functions on $\Bbb R^2$             and prove sharp estimates for them, from which we can deduce  
$L^p$ boundedness of 
maximal functions defined by Fourier multipliers of Bochner-Riesz type 
 on $\Bbb R^2$.  This is a generalization of a result due to A. Carbery 1983. 
\end{abstract}
  \thanks{ 
  \endgraf
  }
\thanks{The author is partly supported
by Grant-in-Aid for Scientific Research (C) No. 20K03651, Japan Society for the  Promotion of Science.}

\maketitle 

\section{Introduction}\label{sec1}

Let  $I$ be a compact interval in $\Bbb R$ and $\psi$  a real valued function 
in $C^\infty(I)$. Let 
$a\in C_0^\infty(\Bbb R^2)$, $\supp (a)\subset I^\circ\times \Bbb R$, 
 where $I^\circ$ denotes the interior of $I$. 
Put 
\begin{equation}\label{e1.1}
\sigma_\lambda(\xi)=a(\xi)(\xi_2-\psi(\xi_1))_+^\lambda
\end{equation}
 for $\lambda>0$, where $r_+=\max(r,0)$ for $r\in \Bbb R$.  Define  
\begin{equation}\label{e1.2}
S^\lambda_Rf(x)=\int_{\Bbb R^2} \sigma_\lambda(R^{-1}\xi) 
\hat{f}(\xi)e^{2\pi i\langle x, \xi\rangle}\, d\xi, \quad R>0, 
\end{equation} 
where $f\in \mathscr{S}(\Bbb R^2)$ (the Schwartz space of infinitely 
differentiable, rapidly decreasing functions), and 
\begin{equation*} 
\hat{f}(\xi)=\mathscr F(f)(\xi) 
= \int_{\Bbb R^2} f(x)e^{-2\pi i\langle x, \xi\rangle}\, dx
\end{equation*} 
is the Fourier  transform on $\Bbb R^2$ with 
$\langle x, \xi\rangle=x_1\xi_1+x_2\xi_2$, 
$x=(x_1,x_2)$, $\xi=(\xi_1, \xi_2)$.   
Put $S^\lambda_*f=\sup_{R>0}|S^\lambda_Rf|$. 
\par 
We assume that 
\begin{enumerate} 
\item[$(A.1)$]  
$0\notin \supp(a)$, 
\item[$(A.2)$]   
$\psi(t)-t\psi'(t)\neq 0$ on $I$, which means that 
the curve $\Gamma=\{(t,\psi(t)): t\in I\}$ does not have a tangent line
 passing through the origin.  
\end{enumerate} 
 In this note we shall prove the following. 
\begin{theorem} \label{thm1.1} 
Suppose that $\psi''\neq 0$ on $I$. Let $\lambda>0$. 
Then, there exists a positive constant $C_\lambda$ such that  
\begin{equation*} 
\|S^\lambda_*f\|_4\leq C_\lambda\|f\|_4. 
\end{equation*} 
\end{theorem}  
This is a generalization of a result of Carbery \cite{Ca}; weighted 
estimates can be found in \cite{Ca2}. 
See Sogge \cite[Chap.2]{Sogg} for related results. 
We refer to \cite{Ca-Sj, Fe2, Sa} for background results.  
\par 
For the $L^2$ estimates we have the following. 
\begin{theorem} \label{thm1.2+} 
Let $\lambda>0$.  
Then,
\begin{equation*} 
\|S^\lambda_*f\|_2\leq C_\lambda\|f\|_2.  
\end{equation*} 
Here we need not assume that $\psi''\neq 0$ on $I$.  
\end{theorem}  
Interpolating between the estimates in Theorems \ref{thm1.1} and 
\ref{thm1.2+}, we have the following. 
\begin{corollary}\label{cor1.3} 
Let $\psi$ satisfy the conditions in Theorem $\ref{thm1.1}$. 
Let $\lambda>0$.  
Then 
\begin{equation*} 
\|S^\lambda_*f\|_p\leq C_\lambda\|f\|_p, \quad 2\leq p\leq 4. 
\end{equation*} 
\end{corollary} 
\par 
Let $\Phi\in C_0^\infty(\Bbb R)$, 
$\supp (\Phi)\subset [-1, 1]$. 
 Let 
$b\in C_0^\infty(\Bbb R^2)$, $\supp (b)\subset I^\circ\times \Bbb R$.  
We assume that $b=0$ near the origin.  
Put 
\begin{equation}\label{e1.3}
\phi(\xi)=\phi^{(\delta)}(\xi)=b(\xi)\Phi(\delta^{-1}(\xi_2-\psi(\xi_1))) 
\end{equation}  
with a small number $\delta \in (0, 1/2]$.     
For appropriate functions $h$ and $f$, define the Littlewood-Paley function 
\begin{equation*} 
g_h(f)=\left(\int_0^\infty |h_t*f|^2\, \frac{dt}{t} \right)^{1/2}, 
\quad h_t(x)=t^{-2}h(t^{-1}x).
\end{equation*}  
\par 
 Let $\eta=\mathscr F^{-1}(\phi^{(\delta)})$. 
To prove Theorems \ref{thm1.1} and \ref{thm1.2+}, we show the following two 
theorems on $g_\eta$.   
\begin{theorem} \label{thm1.2} 
Let $\psi$ be as in Theorem $\ref{thm1.1}$. 
Then we have 
\begin{equation*} 
\|g_\eta(f)\|_4 \leq C\delta^{1/2}\left(\log\frac{1}{\delta}\right)^\tau 
\|f\|_4 
\end{equation*} 
for some $\tau>0$, where the constant $C$ is independent of $\Phi$ provided that $\|(d/dr)^m\Phi\|_\infty \leq 1$ for $0\leq m\leq 3$.     
\end{theorem} 
\begin{theorem} \label{thm1.5+} 
Let $\psi$ be as in Theorem $\ref{thm1.2+}$. 
Then 
\begin{equation*} 
\|g_\eta(f)\|_2 \leq C\delta^{1/2}\|\Phi\|_\infty\|f\|_2,  
\end{equation*} 
 where the constant $C$ is independent of $\Phi$.     
\end{theorem} 

These results can be used to prove the following.   
\begin{theorem} \label{thm1.3} 
Let $r, \lambda, \theta\in \Bbb R$. Define 
$\varphi_{\lambda,\theta}(r)=r^\lambda(\log(2+r^{-1}))^{-\theta}$ 
if $r>0$ and  $\varphi_{\lambda,\theta}(r)=0$ if $r\leq 0$.  
Let $\psi$ be as in Theorem $\ref{thm1.1}$ and 
let $b$ be as in \eqref{e1.3}. Put  
$\Psi_{\lambda,\theta}(\xi)
=b(\xi)\varphi_{\lambda,\theta}(\xi_2-\psi(\xi_1))$. 
Let $\tau$ be as in Theorem $\ref{thm1.2}$. Then  we have 
\begin{equation*} 
\|g_{\mathscr{F}^{-1}(\Psi_{\lambda, \theta})}(f)\|_4 \leq C_{\lambda,\theta}
\|f\|_4    
\end{equation*} 
if $\lambda=-1/2$ and $\theta>\tau+1;$ 
this is also valid for every $\theta\in \Bbb R$ when $\lambda>-1/2$. 
\end{theorem} 
\begin{theorem} \label{thm1.7+} 
Let $\varphi_{\lambda,\theta}$ be as in Theorem $\ref{thm1.3}$.  
Let $\psi$ be as in Theorem $\ref{thm1.2+}$ and 
let $b$ be as in \eqref{e1.3}. Put  
$\Psi_{\lambda,\theta}(\xi)
=b(\xi)\varphi_{\lambda,\theta}(\xi_2-\psi(\xi_1))$. 
Then  
\begin{equation*} 
\|g_{\mathscr{F}^{-1}(\Psi_{\lambda, \theta})}(f)\|_2 \leq C_{\lambda,\theta}
\|f\|_2    
\end{equation*} 
 if $\lambda=-1/2$ and $\theta> 1;$    
this  also holds true  for every $\theta\in \Bbb R$ when $\lambda>-1/2$. 
\end{theorem} 
Theorems \ref{thm1.3} and \ref{thm1.7+} will be shown by applying Theorems 
\ref{thm1.2} and \ref{thm1.5+}, respectively.  
Considering the case when $\lambda>-1/2$ and $\theta=0$ and applying 
interpolation, by Theorems \ref{thm1.3} and \ref{thm1.7+} we 
 have in particular the following. 
\begin{corollary}\label{cor1.8} 
Let $\Psi_{\lambda,\theta}$ be as in Theorem $\ref{thm1.3}$.  Then 
if $\lambda>-1/2$,  we have 
\begin{equation*} 
\|g_{\mathscr{F}^{-1}(\Psi_{\lambda, 0})}(f)\|_p \leq C_{\lambda}
\|f\|_p,   \quad p\in [2, 4].     
\end{equation*} 
\end{corollary}

Theorem \ref{thm1.3} and a modification will be applied in proving Theorem 
\ref{thm1.1}.  Similarly, Theorem \ref{thm1.2+} can be shown by 
Theorem \ref{thm1.7+} and a variant. 
 See Stein-Weiss \cite[\S 5, Chap.\,VII]{SW} for the relation 
between  boundedness of  square functions and  boundedness of 
 maximal Bochner-Riesz type operators.  
Theorems \ref{thm1.2} and \ref{thm1.3} generalize results of \cite{Ca}.  
See also \cite{Cl} and \cite{SZ} for related results.   
A result analogous to Theorem \ref{thm1.3} of the case $\lambda>-1/2$ and 
$\theta=0$ is stated in \cite{Cl} by considering square functions defined with 
 homogeneous functions in a general setting.  
\par 
In the proofs of the theorems we shall apply arguments of \cite{Ca}. Also, 
we shall apply a reasoning in \cite{Co2}, which leads to \eqref{e3.5} below.  
See also \cite{ Co, Fe} for background results in the ideas of the proofs. 
One purpose of this note is to provide proofs of 
the theorems with computations being accessible based on the condition 
$(A.2)$ and \cite[Lemma 5.3]{Sa}.       
\par  
By a partition of unity, the  proofs of Theorems \ref{thm1.1}, \ref{thm1.2} and \ref{thm1.3} may be divided into several cases:  
\par 
\begin{enumerate} 
\item[$(B.1)$]   $\Gamma \subset (-b,b)\times (c,d)$, $0<b$, $0<c<d$, 
$I=[A,B]\subset (-b,b)$;
\item[$(B.2)$]  $\Gamma \subset (-b,b)\times (-d, -c)$, $0<b$, $0<c<d$, 
$I=[A,B]\subset (-b,b)$; 
\item[$(B.3)$] $\Gamma \subset (a,b)\times (-d, d)$, $0<a<b$, $0<d$, 
$I=[A,B]\subset (a,b)$;
\item[$(B.4)$] $\Gamma \subset (-b,-a)\times (-d, d)$, $0<a<b$, $0<d$, 
$I=[A,B]\subset (-b,-a)$. 
\end{enumerate}  
In Section \ref{sec2} we shall review some preliminary results for Kakeya 
maximal functions and Littlewood-Paley inequalities.   
In Section \ref{sec3} we shall prove Theorems \ref{thm1.2} and \ref{thm1.3} 
under the conditions of case $(B.1)$ with $\Psi'>0$ on $I$, where 
$\Psi(t)=t/\psi(t)$; we note that the condition $\psi(t)\neq t\psi'(t)$ 
implies that $\Psi'(t)\neq 0$. 
We shall prove a vector valued inequality for a sequence of Fourier multiplier 
operators in Section \ref{sec3} (Proposition \ref{prop3.7}), which will be 
applied to prove Theorem \ref{thm1.2}.  
To prove Theorem \ref{thm1.2} we shall also need some results on geometry of 
support sets of certain Fourier multipliers,
 which will be taken from \cite{Sa}.  
Theorem \ref{thm1.3} follows from Theorem \ref{thm1.2}. 
\par 
Theorem \ref{thm1.1} will be shown in Section \ref{sec4} by applying 
Theorem \ref{thm1.3} under the conditions of the 
case $(B.1)$ with $\Psi'>0$ on $I$. To apply Theorem \ref{thm1.3} we need to 
introduce a homogeneous function of degree one in a cone. The function will be 
constructed by using the condition $(A,2)$. (See Corollaries \ref{cor4.2} and 
\ref{cor4.4}.)  
\par 
 Theorems \ref{thm1.1}, \ref{thm1.2} and \ref{thm1.3} will be shown in 
Section \ref{sec5} in their full generalities claimed. The proof of 
Theorem \ref{thm1.5+} will be given in Section \ref{sec6}. 
The other results on $L^2$ boundedness can be shown in the same way as in 
the case of $L^4$ boundedness by applying the boundedness of $g_\eta$, 
as we can easily see.

\section{Maximal functions and  Littlewood-Paley inequalities}\label{sec2}  
 
Let $N$ be a positive integer and 
let $\Omega_N$ be a set of vectors $v$ in $\Bbb R^2$ with $|v|=1$ 
such that $\card \Omega_N\leq N$.  
Let $\mathscr B_N$ be the set of all rectangles in $\Bbb R^2$ 
 one of whose sides is parallel to a vector in $\Omega_N$. 
Let $M_{\Omega_N}$ be the maximal operator defined by  
\begin{equation*} 
M_{\Omega_N}f(x)
=\sup_{x\in R\in \mathscr B_N} \frac{1}{|R|}\int_R|f(y)|\, dy,  
\end{equation*} 
where $|R|$ denotes the Lebesgue measure of $R$. 
Then we have the following (see \cite{Stro, Ka}).
\begin{lemma}\label{lem2.1}  
There exists a positive constant $\alpha$ independent of $N$ such that 
\begin{equation*} 
\|M_{\Omega_N}f\|_2 \leq C(\log(2N))^\alpha\|f\|_2  
\end{equation*} 
for some $C>0$.  
\end{lemma} 
\par 
Let $P_n=I_n\times \Bbb R$, $n \in \Bbb Z$ (the set of integers), where 
$\{I_n\}$ is a sequence of non-overlapping intervals in $\Bbb R$   
such that $|I_n|=\tau$, $\forall n$.     
Let $T_mf=\mathscr F^{-1}(m\hat{f})$ for a bounded function $m$. 
If $m=\chi_E$  for a measurable set $E$, we also write $T_{E}$ for 
$T_{\chi_E}$. Then we have the following result 
(see \cite[Lemma $2$]{Ca}). 

\begin{lemma}\label{lem2.2}  
Let $w$ be a non-negative function on $\Bbb R^2$.   Let $s>1$. Then there exists $\Theta>0$ such that 
\begin{equation*} 
\int_{\Bbb R^2} \sum_{n\in \Bbb Z} |T_{P_n}f|^2 w\, dx\leq 
C\left(\frac{1}{s-1}\right)^{\Theta} \int_{\Bbb R^2} 
|f|^2(M_1(w^s))^{1/s}\, dx, 
\end{equation*} 
where $M_1$ denotes the strong maximal operator and $C$ is a constant 
independent of $\tau$.  This is also true if $P_n$ is replaced by 
$Q_n=\Bbb R\times I_n$. 
\end{lemma} 
\begin{proof} 
Let $H$ be the Hilbert transform on $\Bbb R$. Put 
$M^{(s)}(w)= (M(w^s))^{1/s}$, $s>1$,  where $M$ 
denotes the Hardy-Littlewood maximal operator on $\Bbb R$ and 
 $w$ a non-negative function on $\Bbb R$.
Then it is known that 
\begin{equation} \label{e2.1}
\int |Hf|^2M^{(s)}(w)\, dx \leq C\left(\frac{1}{s-1}\right)^{\beta_0}\int |f|^2
M^{(s)}(w)\, dx 
\end{equation} 
for some $\beta_0>0$. It follows that if $J=[a,b]$ is an interval 
\begin{equation} \label{e2.2} 
\int |T_J f|^2M^{(s)}(w)\, dx \leq C\left(\frac{1}{s-1}\right)^{2\beta_0}
\int |f|^2 M^{(s)}(w)\, dx .  
\end{equation} 
This can be seen from \eqref{e2.1} by the equation 
\begin{equation*} 
T_J=(1/2)\left(I+iN_aHN_{-a}\right)
(1/2)\left(I-iN_bHN_{-b}\right), 
\end{equation*} 
where $I$ denotes the identity operator and $N_af(x)=e^{2\pi ixa}f(x)$.   
\par 
We show the lemma with $\Theta=2\beta_0 +1$. 
Let $J_j=[(j-1/2)\tau, (j+1/2)\tau]$, $j\in \Bbb Z$. 
If $\varphi\in C^\infty(\Bbb R)$, $\supp(\varphi) \subset [-2, 2]$, 
$\varphi=1$ on $[-1,1]$ and  
$\widehat{f_j}(\xi)=\varphi(2j-2\tau^{-1}\xi)\hat{f}(\xi)$, then 
$T_{J_m}f=T_{J_m}f_m$ and 
\begin{align*}  
|T_{I_j}f|^2&= \left|\sum_{m: |J_m\cap I_j|\neq 0}T_{J_{m}\cap I_j}f\right|^2 
\leq 2\sum_{m: |J_m\cap I_j|\neq 0}\left|T_{J_{m}\cap I_j}f\right|^2 
\\ 
&=2\sum_{m: |J_m\cap I_j|\neq 0}\left|T_{J_{m}\cap I_j}f_m\right|^2. 
\end{align*} 
Using this and \eqref{e2.2} with $J_m\cap I_j$ in place of $J$, we have ,    
\begin{align*} 
\sum_j\int |T_{I_j}f|^2 M^{(s)}(w)\, dx &\leq 
C\sum_j\sum_{m: |J_m\cap I_j|\neq 0}
\left(\frac{1}{s-1}\right)^{2\beta_0}\int \left|f_m\right|^2 M^{(s)}(w)\, dx 
\\ 
&=C\sum_m\sum_{j: |J_m\cap I_j|\neq 0}
\left(\frac{1}{s-1}\right)^{2\beta_0}\int \left|f_m\right|^2 M^{(s)}(w)\, dx 
\\ 
&\leq C\left(\frac{1}{s-1}\right)^{2\beta_0}\sum_m  
\int |f_{m}|^2 M^{(s)}(w)\, dx 
\\ 
&\leq C\left(\frac{1}{s-1}\right)^{2\beta_0}
\int |f|^2 MM^{(s)}(w)\, dx   
\\ 
&\leq C\left(\frac{1}{s-1}\right)^{2\beta_0+1} 
\int |f|^2 M^{(s)}(w)\, dx,    
\end{align*} 
where the penultimate inequality follows from the arguments in 
\cite[pp. 297-298]{Co4} and the last inequality is implied by 
the well-known estimate for the $A_1$ constant of $M^{(s)}(w)$. 
 From this the conclusion of the Lemma \ref{lem2.2} can be derived. 
\end{proof}

\section{Proofs of Theorems \ref{thm1.2} and \ref{thm1.3} 
under $(B.1)$ with $\Psi'>0$} 
\label{sec3} 
In this section we prove Theorems \ref{thm1.2} and \ref{thm1.3}
 under the conditions of $(B.1)$ stated in Section \ref{sec1} 
with $\Psi'>0$.   
We may assume that $\delta=2^{-L}$ for some positive integer $L$.  
 Let 
\begin{equation*}
W_k=\cup \{(2^n,2^{n+1}]: n\equiv k \mod L\}, \quad k= 0, 1, \dots, L-1. 
\end{equation*}
We have 
\begin{equation} \label{e3.1}
g_h(f)\leq \sum_{k=0}^{L-1} g_h^{(k)}(f), 
\end{equation} 
where 
\begin{equation*} 
g_h^{(k)}(f)=\left(\int_{W_k}
|h_t*f(x)|^2\, \frac{dt}{t}\right)^{1/2}. 
\end{equation*} 
We write $W$ for $W_0$ and estimate $g_\eta^{(0)}(f)$.   
The other functions $g_\eta^{(k)}(f)$ will be estimated similarly. 
\par 
 We assume that $\delta$ is a small positive number and 
that $\delta^{-1/2}$ is an integer. Put $N=c_0\delta^{-1/2}$, where 
$c_0\delta^{-1/2}$ is a positive integer.  
We assume that 
$\omega_k=[a_{k-1}, a_k]$, $-b=a_0<a_1<\dots <a_{N}=b$, 
$a_k-a_{k-1}=\delta^{1/2}$, $1\leq k\leq N$, 
 $\cup_{k=1}^N \omega_k=[-b,b]$ and that     
 $H_j=[b_{j-1},b_j]$, $b_j-b_{j-1}=\delta^{1/2}$, $1\leq j\leq N$, 
$\cup_{j=1}^NH_j=[c,d]$. 
\par 
Recall that  
\begin{equation*} 
\phi^{(\delta)}(\xi)=\phi(\xi)= b(\xi)\Phi(\delta^{-1}(\xi_2-\psi(\xi_1))).  
\end{equation*} 
If $\delta$ is small enough, we have 
\begin{equation*} 
\supp(\phi)\subset \left(\bigcup_{\omega_\ell\subset I} \omega_\ell\right)
\times \Bbb R. 
\end{equation*} 
Decompose $\phi^{(\delta)}$ as 
\begin{equation} \label{e3.2} 
\phi^{(\delta)}=\sum_\ell \zeta_\ell\phi^{(\delta)} +
\sum_\ell \widetilde{\zeta}_\ell\phi^{(\delta)}
=: \phi^{(\delta)}_{(1)} +\phi^{(\delta)}_{(2)}, 
\end{equation} 
where $\zeta_\ell, \widetilde{\zeta}_\ell\in C_0^\infty(\Bbb R)$ such that 
$\supp(\zeta_\ell)\subset \omega_\ell$, $\supp(\widetilde{\zeta_\ell})\subset 
\omega_\ell+(1/2)\delta^{1/2}$ and  
\begin{equation*} 
|(d/ds)^\gamma \zeta_\ell|\leq 
C_\gamma \delta^{-\gamma/2}, \quad |(d/ds)^\gamma \widetilde{\zeta}_\ell|\leq 
C_\gamma \delta^{-\gamma/2}, \quad \forall \gamma \in \Bbb Z\cap [0,\infty).  
\end{equation*} 
We estimate $g_{(i)}(f)$ for $i=1$, where 
$g_{(i)}(f)=g_{h_{(i)}}(f)$ with 
$h_{(i)}=\mathscr{F}^{-1}(\phi^{(\delta)}_{(i)})$, $i=1, 2$. 
 The function $g_{(2)}(f)$ can be estimated similarly. 
Also, let $g_{(i)}^{(k)}(f)=g_{h_{(i)}}^{(k)}(f)$, $0\leq k\leq L-1$, 
$i=1, 2$. 
\par 
To prove Theorem \ref{thm1.2}, we apply Lemma 5.3 of \cite{Sa} (see \eqref{e3.5} below). For this 
reason it is convenient to assume that $\supp(\Phi)\subset [0, c_*]$, where 
$c_*$ is a sufficiently small positive number depending on $\psi$.  The 
result for $\Phi$ with support in $[-1, 1]$ easily follows by change of 
variables.  
We divide the intervals $\{\omega_\ell\}$ into $4$ families as in \cite{Sa}: 
\begin{equation*}   
\begin{split} 
\mathscr F_1&= \{\omega_1, \omega_5, \omega_9, \dots \}, \quad  
\mathscr F_2= \{\omega_2, \omega_6, \omega_{10}, \dots \}, 
\\ 
\mathscr F_3&= \{\omega_3, \omega_7, \omega_{11}, \dots \}, \quad  
\mathscr F_4= \{\omega_4, \omega_8, \omega_{12}, \dots \}.   
\end{split} 
\end{equation*}
We focus on the family $\mathscr F_1$; the other families can be treated 
similarly.    
Let $\zeta_\ell$, $\omega_\ell \subset I$, be as above. 
Let $\nu(\ell)=4(\ell-1)+1$. 
Put 
\begin{gather*} \label{e3.2+}
s^{(\delta)}(\xi)=\sum_{\ell: \omega_{\nu(\ell)} \subset I} 
s^{(\delta)}_{\nu(\ell)}(\xi), 
\quad        
s^{(\delta)}_\ell(\xi)= \phi(\xi)\zeta_\ell(\xi_1), \quad 
s^{(\delta,t)}_\ell(\xi)=s^{(\delta)}_\ell(t\xi), 
\\ \label{e3.3+}
s^{(\delta,t)}(\xi)=s^{(\delta)}(t\xi). 
\end{gather*} 
\par 
We consider $g_{\mathscr F_1}^{(k)}(f):=
g_{\mathscr F^{-1}(s^{(\delta)})}^{(k)}(f)$, $0\leq k\leq L-1$ and  
we define $g_{\mathscr F_i}^{(k)}(f)$, $2\leq i\leq4$, similarly to
 $g_{\mathscr F_1}^{(k)}(f)$ by using $\mathscr F_i$. 
We prove 
\begin{equation} \label{e3.3} 
\|g_{\mathscr F_i}^{(k)}(f)\|_4 \leq 
C\delta^{1/2}\left(\log\frac{1}{\delta}\right)^{(\beta/2)+2\alpha}\|f\|_4, 
\quad  \beta=3\Theta+2\beta_0,     
\end{equation} 
for $1\leq i\leq 4$, $0\leq k\leq L-1$, 
where $\Theta$ and $\beta_0$ are as in Lemma $\ref{lem2.2}$ and \eqref{e2.1}, 
respectively, and $\alpha$ is as in Lemma $\ref{lem2.1}$.     
It follows that 
\begin{equation*}
\|g_{(i)}^{(k)}(f)\|_4 \leq 
C\delta^{1/2}\left(\log\frac{1}{\delta}\right)^{(\beta/2)+2\alpha}\|f\|_4    
\end{equation*} 
for $i=1$.  
By this and \eqref{e3.1} we see that 
\begin{equation*}
\|g_{(i)}(f)\|_4 \leq 
C\delta^{1/2}\left(\log\frac{1}{\delta}\right)^{(\beta/2)+2\alpha+1}\|f\|_4   
\end{equation*} 
for $i=1$.  This estimate for $i=2$ can be shown similarly. Thus we have 
\begin{equation} \label{e3.4} 
\|g_\eta(f)\|_4 \leq C\delta^{1/2}
\left(\log\frac{1}{\delta}\right)^{(\beta/2)+2\alpha+1}\|f\|_4.  
\end{equation}  
This will complete the proof of Theorem \ref{thm1.2} with 
$\tau=(\beta/2)+2\alpha+1$ 
under the conditions of the case $(B.1)$ with $\Psi'>0$. 
\par 
Now we prove \eqref{e3.3} for $k=0$ and $i=1$; the other estimates can be 
shown similarly. 
We observe that if $t, u\in W$ and $t>u$,  then $t<2u$ or $u<2\delta t$.   
Put 
\begin{equation*} 
\Xi=\{(u,t)\in W\times W: u<t<2u\}, \quad 
\Lambda=\{(u,t)\in W\times W: u<2\delta t\}.  
\end{equation*} 
 By applying the Plancherel theorem we see that 
\begin{align*} 
\|g_{\mathscr F_1}^{(0)}(f)\|_4^4
&\leq 2\iint_{W\times W, t>u} 
\int_{\Bbb R^2}\left|s^{(\delta,t)}\hat{f}*s^{(\delta,u)}\hat{f}\right|^2
\, d\xi \, \frac{dt}{t}\, \frac{du}{u} 
\\ 
&= 2\iint_{\Xi} 
\int_{\Bbb R^2}\left|s^{(\delta,t)}\hat{f}*s^{(\delta,u)}\hat{f}\right|^2\, d\xi \, \frac{dt}{t}\, \frac{du}{u} 
\\ 
&\quad + 2\iint_{\Lambda} 
\int_{\Bbb R^2}\left|s^{(\delta,t)}\hat{f}*s^{(\delta,u)}\hat{f}\right|^2
\, d\xi \, \frac{dt}{t}\, \frac{du}{u} 
\\ 
&= 2\iint_{\Xi} 
\int_{\Bbb R^2}\left|\sum_{\ell}\sum_{m} s^{(\delta,t)}_{\nu(\ell)}  \hat{f}*
s^{(\delta,u)}_{\nu(m)} \hat{f}\right|^2\, d\xi \, \frac{dt}{t}\, \frac{du}{u} 
\\ 
&\quad + 2\iint_{\Lambda} 
\int_{\Bbb R^2}\left|
\sum_{\ell} s^{(\delta,u)}_{\nu(\ell)} 
\hat{f}*s^{(\delta,t)}\hat{f}\right|^2\, d\xi \, \frac{dt}{t}\, \frac{du}{u} 
\\ 
&=: Q_1+Q_2. 
\end{align*} 
Arguing similarly to \cite[pp.\,319-320]{Sa} with application of 
\cite[Lemma 5.3]{Sa} for $Q_1$, where we need the condition that 
$\psi''\neq 0$, we see that 
\begin{align} \label{e3.5} 
Q_1+Q_2 &\leq C\iint_{\Xi} 
\int_{\Bbb R^2} \sum_{\ell}\sum_{m}\left| s^{(\delta,t)}_{\nu(\ell)} 
\hat{f}*s^{(\delta,u)}_{\nu(m)}  \hat{f}\right|^2\, d\xi 
\, \frac{dt}{t}\, \frac{du}{u} 
\\ 
&\quad + C\iint_{\Lambda} 
\int_{\Bbb R^2}\sum_{\ell}\left| s^{(\delta,u)}_{\nu(\ell)} 
\hat{f}*s^{(\delta,t)}\hat{f}\right|^2\, d\xi \, \frac{dt}{t}
\, \frac{du}{u}           \notag
\end{align} 
(see also \cite{Co2} for these arguments). 
Define 
\begin{equation*} 
\mathscr F(S^{(\delta,t)}_\ell f)(\xi)= s^{(\delta,t)}_\ell(\xi) \hat{f}(\xi), 
\quad 
V(f)(x)
=\left(\int_0^\infty\sum_{\ell} |S^{(\delta,t)}_\ell f|^2\, 
\frac{dt}{t}\right)^{1/2}.  
\end{equation*}
Applying the Plancherel theorem on the right hand side of \eqref{e3.5}, we 
see that 
\begin{equation*}
\|g_{\mathscr F_1}^{(0)}(f)\|_4^4  
\leq  C\int V(f)(x)^4 \, dx + C\int \left(g_{\mathscr F_1}^{(0)}(f)(x)V(f)(x)\right)^2\, dx,
\end{equation*} 
which implies that 
\begin{equation}\label{e3.6}
\|g_{\mathscr F_1}^{(0)}(f)\|_4\leq  C\|V(f)\|_4. 
\end{equation} 
\par 
Let $[a_{\ell-2},a_{\ell+1}]\subset I$.   Define 
\begin{gather*}   
\Delta_\ell=\{(\xi_1,\xi_2): \Psi(a_{\ell-1})\xi_2\leq \xi_1 
\leq \Psi(a_{\ell})\xi_2, \, \xi_2>0\}
=\{r(t, \psi(t)): t\in \omega_\ell, r>0\},  
\\ 
S^\ell_h=\{(\xi_1,\xi_2): \psi'(a_{\ell-1})\xi_1 +h\delta \leq \xi_2 
\leq \psi'(a_{\ell-1})\xi_1 +(h+1)\delta \}, \quad h\in \Bbb Z, 
\\ 
P^{\ell}_h=S^\ell_h\cap \widetilde{\Delta}_\ell, 
\quad 
P^{\ell,k,j}_h=S^\ell_h\cap (\omega_k\times H_j)\cap \widetilde{\Delta}_\ell 
\cap (\Bbb R\times [c,d]), \quad 
\widetilde{\Delta}_\ell=\cup_{|\ell'-\ell|\leq 1} \Delta_{\ell'}. 
\end{gather*}  
To prove \eqref{e3.4}, we may assume that $\supp(\hat{f}) 
\subset \cup_{1\leq \ell\leq N}\Delta_\ell$. 
To estimate $V(f)$ we need Lemmas \ref{lem3.2} and \ref{lem3.3} below. 
We first observe the following. 
\begin{lemma} \label{lem3.1} 
It holds that $\supp(s^{(\delta)}_\ell)\subset \widetilde{\Delta}_\ell$.  
\end{lemma} 
\begin{proof}   
(1) Suppose that $a_{\ell-1}\geq 0$. Recall that $\Phi$ is supported in 
$[0,c_*]$, $c_*<1$.    
 If $\xi \in  
\supp(s^{(\delta)}_\ell)\setminus \Delta_\ell$, then 
$\xi=(s,u+\kappa)$ with $s/u=\Psi(a_{\ell-1})$ for some 
$s\in [a_{\ell-1},a_\ell]$ and $\kappa \in [0,\delta]$. 
\par 
Note that $|s/(u+\kappa)-s/u|\leq C\delta$ and 
that $|\Psi(a_{\ell-1})-\Psi(a_{\ell-2})| \geq c\delta^{1/2}$, 
$s/u=\Psi(a_{\ell-1})$.  Thus 
 we see that 
$s/(u+\kappa)\geq \Psi(a_{\ell-2})$ if $\delta$ is small enough, 
and hence  $\Psi(a_{\ell-2})\leq s/(u+\kappa)\leq \Psi(a_{\ell-1})$, 
which implies that $\xi=(s,u+\kappa)\in \Delta_{\ell-1}$. 
\par 
(2) Suppose that $a_{\ell}\leq 0$.  
If $\xi \in  
\supp(s^{(\delta)}_\ell)\setminus \Delta_\ell$, then $\xi
=(s,u+\kappa)$ with $s/u=\Psi(a_\ell)$ 
for some 
$s\in [a_{\ell-1},a_\ell]$ and $\kappa \in [0,\delta]$. 
Therefore, arguing as in part (1), we see that $\xi\in \Delta_{\ell+1}$ 
 if $\delta$ is small enough.  
\par 
(3) Suppose that $a_{\ell-1}< 0<a_\ell$.  Then we can see that 
$\supp(s^{(\delta)}_\ell)\setminus \Delta_\ell= \emptyset$.  
\par 
Combining results in (1), (2) and (3), we finish the proof of the lemma. 
\end{proof}

\begin{lemma} \label{lem3.2} 
Let $t\in [1,2]$ and  
\begin{equation*} 
\mathscr F_{\ell, t}=\{h\in \Bbb Z: S^\ell_h\cap \supp(s^{(\delta,t)}_\ell)\neq \emptyset\}, 
\end{equation*} 
\begin{equation*} 
\mathscr G_{\ell, t}=\{(k,j): (\omega_k\times H_j)\cap 
\supp(s^{(\delta,t)}_\ell)\neq \emptyset\}. 
\end{equation*} 
Then 
$\card(\mathscr F_{\ell, t})\leq C$, $\card(\mathscr G_{\ell, t})\leq C$
 with a constant $C$ independent of $\ell$, $t$ and 
\begin{equation*} 
s^{(\delta,t)}_\ell=\sum_{(k,j) \in \mathscr G_{\ell, t}}\, 
\, \sum_{h\in \mathscr F_{\ell, t}} \chi_{P^{\ell,k,j}_h} 
s^{(\delta,t)}_\ell.  
\end{equation*} 
\end{lemma} 
\begin{proof}  
Since $\supp(s^{(\delta,t)}_\ell)\subset B(t^{-1}(a_{\ell-1},\psi(a_{\ell-1})), c\delta^{1/2})$ for some $c>0$, we have $\card(\mathscr G_{\ell, t})\leq C$, 
where $B(\xi,r)$ denotes a ball of radius $r$ centered at $\xi$.     
\par 
The slope of the tangent line 
to the curve $t^{-1}\Gamma$ at $t^{-1}(a_{\ell-1}, 
\psi(a_{\ell-1}))$ is $\psi'(a_{\ell-1})$. 
Thus $\supp(s^{(\delta,t)}_\ell)$ is  contained in a parallelogram centered 
at  $t^{-1}(a_{\ell-1}, \psi(a_{\ell-1}))$ with side lengths $c\delta$ and 
$c\delta^{1/2}$, where the longer sides are parallel to the vector 
$(1, \psi'(a_{\ell-1}))$. 
From this we can deduce that $\card(\mathscr F_{\ell, t})\leq C$. 
\par 
By Lemma \ref{lem3.1}, 
$\supp(s^{(\delta,t)}_\ell)=t^{-1}\supp(s^{(\delta)}_\ell) 
\subset t^{-1}\widetilde{\Delta}_\ell=\widetilde{\Delta}_\ell$.  
Also, $\supp(s^{(\delta,t)}_\ell)\subset \cup_{(k,j)\in \mathscr G_{\ell, t}}
(\omega_k\times H_j)$,    $\supp(s^{(\delta,t)}_\ell)\subset \cup_{h \in 
\mathscr F_{\ell, t}}S^\ell_h$. These observations prove the equation for 
decomposing $s^{(\delta,t)}_\ell$, since $P^{\ell,k,j}_h= 
(\omega_k\times H_j)\cap S^\ell_h \cap \widetilde{\Delta}_\ell$ and 
the sets $P^{\ell,k,j}_h$ are mutually non-overlapping when $\ell$ is fixed. 
\end{proof} 
\begin{lemma}\label{lem3.3}  
Let 
\begin{align*} 
E(n,\ell,h)&= \{t\in [2^n,2^{n+1}]: (2^{-n}P^{\ell}_h) \cap 
\supp(\phi^{(\delta,t)}) \neq \emptyset\} 
\\ 
&=\{t\in [2^n,2^{n+1}]: (t2^{-n}P^{\ell}_h) \cap 
\supp(\phi^{(\delta)}) \neq \emptyset\},   
\end{align*}  
where $\phi^{(\delta,t)}(\xi)=\phi^{(\delta)}(t\xi)$. 
Then there exists a constant $C$ independent of $n,\ell,h$ such that   
\begin{equation*} 
\int_0^\infty \chi_{E(n,\ell,h)}(t) \, \frac{dt}{t} \leq C\delta. 
\end{equation*}  
\end{lemma}  
The proof of this requires the following results; Lemmas \ref{lem3.4}, 
\ref{lem3.5} and \ref{lem3.6} below. 
In Lemmas \ref{lem3.4} and \ref{lem3.5}, we consider the cases $(B.i)$, 
$1\leq i\leq 4$, together assuming $(A.2)$.  
\begin{lemma}\label{lem3.4} 
Let $I'$ be a compact interval such that $I'\subset I^\circ$. 
Then there exists a positive constant $B_1$ such that 
\begin{equation*} 
\{(\xi_1,\xi_2): |\xi_2-\psi(\xi_1)|\leq \delta, \, \xi_1\in I'\}
\subset \cup\{s\Gamma : |1-s|\leq B_1\delta\},  
\end{equation*}
if $\delta$ is small enough. 
\end{lemma} 
\begin{proof} 
Here we give the proof for the case $(B.1)$. The other cases can be handled 
similarly. 
Let $t\in I'$ and $\kappa \in [-\delta, \delta]$. Let 
$m_0=\min(\Psi(A), \Psi(B))$, $m_1=\max(\Psi(A), \Psi(B))$. 
Then $m_0+\epsilon_0<t/\psi(t)<m_1-\epsilon_0$ for some $\epsilon_0>0$.  
(In the cases $(B.3)$, $(B.4)$, we use $\Psi_*(t)=\psi(t)/t$ instead of 
$\Psi$.)  
Thus, if $\delta$ is small enough, we have $t/(\psi(t)+\kappa)\in [m_0, m_1]$. 
So, by the intermediate value theorem there exists $s\in I=[A,B]$ such that 
$\Psi(s)=s/\psi(s)= t/(\psi(t)+\kappa)$. 
\par 
If $s\neq 0$, put 
$r=t/s=(\psi(t)+\kappa)/\psi(s)$. Then  we have 
$r(s,\psi(s))=(t,\psi(t)+\kappa)$ and  
\begin{equation*} 
\left|\frac{t}{\psi(t)}-\frac{s}{\psi(s)}\right|=
\left|\frac{t\kappa}{\psi(t)(\psi(t)+\kappa)}\right| \leq C\delta. 
\end{equation*} 
Thus we have $|t-s|\leq C\delta$, since 
$c_1\leq |\Psi'|\leq c_2$ with positive constants $c_1, c_2$, which 
follows by the assumption that $\Psi'\neq 0$ on $I$.  This implies that 
\begin{equation*} 
(s, \psi(s)), (t,\psi(t)+\kappa)\in B((t,\psi(t)), c\delta).  
\end{equation*} 
It follows that 
\begin{equation*} 
\left|(s, \psi(s))- r(s,\psi(s))\right|=|1-r|\left|(s,\psi(s))\right| 
\leq C\delta, 
\end{equation*} 
and hence $|1-r|\leq C\delta$. This proves the assertion of the lemma when 
$s\neq 0$. 
\par 
If $s=0$, then $t=0$ and $R(0,\psi(0))=(0,\psi(0)+\kappa)$ with $R=
(\psi(0)+\kappa)/\psi(0)=1+\kappa/\psi(0)$, $|\kappa/\psi(0)|\leq C\delta$.  
This applies to the case $s=0$. 
\end{proof} 

\begin{lemma} \label{lem3.5} 
If $r, s>0$ and $r\neq s$, then $r\Gamma \cap s\Gamma =\emptyset$. 
\end{lemma} 
\begin{proof} 
We prove the lemma for the case $(B.1)$. The results in the other cases 
can be shown analogously. 
The proof is by contradiction.  Suppose that $r\Gamma \cap s\Gamma \neq  
\emptyset$ for some $r, s>0$ with $r\neq s$. 
We may assume that $r\Gamma \cap \Gamma \neq \emptyset$ for some $r>1$. 
Then $(t,\psi(t))=r(u,\psi(u))$ for some $t, u \in I$. 
We see that $u\neq 0$, since if $u=0$, then $t=0$ and $\psi(0)=r\psi(0)$, 
which implies that $r=1$. So, $u\neq 0$, $u\neq t$ and  $\Psi(t)=\Psi(u)$. 
By Rolle's theorem there exists 
$t_0$ between $t$ and $u$ such that $\Psi'(t_0)=0$. This contradicts 
the assumption $(A.2)$, since $\Psi'(t)=(\psi(t)-t\psi'(t))/\psi(t)^2$.   
\end{proof} 

\begin{lemma} \label{lem3.6}  
Let $\tau \in [1/2, 1]$. If $P^{\ell}_h\cap (\tau \supp(\phi^{(\delta)}))\neq 
\emptyset$, then there exists a positive constant $D$ such that  
\begin{equation*} 
P^{\ell}_h \subset \cup\{s\tau\Gamma : 1-D\delta \leq s\leq 1+D\delta\}.  
\end{equation*} 
\end{lemma} 
\begin{proof} 
We have $\tau^{-1}P^{\ell}_h\cap (\supp(\phi^{(\delta)}))\neq \emptyset$. 
So by Lemma \ref{lem3.4},  $\tau^{-1}P^{\ell}_h\cap v\Gamma \neq \emptyset$ 
for some $v\in [1-B_1\delta, 1+B_1\delta]$. Thus 
 $u^{-1}P^{\ell}_h\cap \Gamma \neq \emptyset$ with $u=v\tau$. We note that 
\begin{equation*} 
u^{-1}S^\ell_h=\{(\xi_1,\xi_2): \psi'(a_{\ell-1})\xi_1+u^{-1}h\delta 
\leq \xi_2 \leq \psi'(a_{\ell-1})\xi_1+u^{-1}(h+1)\delta\}. 
\end{equation*} 
So, there exists $h_1\in [h,h+1]$ such that, if $L_{h_1}$ is the line defined 
by the equation $\xi_2=\psi'(a_{\ell-1})\xi_1+u^{-1}h_1\delta$, 
$\Gamma\cap u^{-1}\widetilde{\Delta}_\ell\cap L_{h_1}=\Gamma\cap 
\widetilde{\Delta}_\ell\cap L_{h_1} 
\neq \emptyset$. 
\par 
 Let $l_{a_{\ell-1}}$ be the tangent line to $\Gamma$ at 
$(a_{\ell-1}, \psi(a_{\ell-1}))$. Let $\Gamma_\ell=\Gamma\cap
\widetilde{\Delta}_\ell=
\{(t,\psi(t)): t\in[a_{\ell-2}, a_{\ell+1}]\}$ and $l_{a_{\ell-1}}^*= 
l_{a_{\ell-1}}\cap \widetilde{\Delta}_\ell$. Then we see that 
\begin{equation*} 
l_{a_{\ell-1}}^* \subset \Gamma_\ell(c_1):=\{(\xi_1,\xi_2)\in 
\widetilde{\Delta}_\ell: |\xi_2-\psi(\xi_1)|<c_1\delta, \, 
|a_{\ell-1}-\xi_1|\leq c_1\delta^{1/2}\} 
\end{equation*} 
for some $c_1>0$. 
Here we give the proof of this. 
Let $(\alpha_k, \beta_k)$ be the point of 
intersection of lines $l_{a_{\ell-1}}+(0,d)$, $|d|\leq c\delta$, 
and $a_k\xi_2=\psi(a_k)\xi_1$, 
$k=\ell+1$, $\ell-2$. 
It suffices to show that $|\alpha_k-a_{\ell-1}|\leq c\delta^{1/2}$ for some 
$c>0$.  For this we note that 
\begin{align}\label{e3.7}
\alpha_k-a_{\ell-1}
&=\frac{a_k\psi(a_{\ell-1})-\psi(a_k)a_{\ell-1}+a_kd}
{\psi(a_k)-a_k\psi'(a_{\ell-1})} 
\\ 
&=\frac{(a_k-a_{\ell-1})\psi(a_{\ell-1})+ 
a_{\ell-1}(\psi(a_{\ell-1})-\psi(a_k))+a_kd}
{\psi(a_k)-a_k\psi'(a_{k})+a_k(\psi'(a_k)-\psi'(a_{\ell-1}))}. \notag  
\end{align} 
From this with $d=0$ we have the inequality claimed, since 
$|\psi(a_k)-a_k\psi'(a_{k})|\geq c>0$ by $(A.2)$. 
\par 
Similarly we have 
\begin{equation*} 
\Gamma_\ell\subset l_{a_{\ell-1}}(c_2):= \{(\xi_1,\xi_2)\in 
\widetilde{\Delta}_\ell: 
\exists \xi_2' \,\, \text{such that} \,\, |\xi_2-\xi_2'|<c_2\delta, (\xi_1,\xi_2')\in l_{a_{\ell-1}}\} 
\end{equation*} 
for some $c_2>0$.  Since $\Gamma\cap \widetilde{\Delta}_\ell\cap L_{h_1} 
\neq \emptyset$ and 
\begin{equation*} 
\Gamma\cap \widetilde{\Delta}_\ell\cap L_{h_1} \subset \Gamma_\ell \subset 
l_{a_{\ell-1}}(c_2) 
\end{equation*} 
and the slopes of the lines  $L_{h_1}$ and $l_{a_{\ell-1}}$ are the same, the 
distance between $\xi_2$-intercepts of the lines $L_{h_1}$ and 
$l_{a_{\ell-1}}$ is
less than $c\delta$. Thus $u^{-1}S^\ell_h\cap \widetilde{\Delta}_\ell \subset 
l_{a_{\ell-1}}(c_3)$ for some $c_3>0$ and hence  
\begin{equation*} 
u^{-1}P^{\ell}_h \subset l_{a_{\ell-1}}(c_3)\subset \Gamma_\ell(c_4)
\end{equation*} 
for some $c_4>0$, where the second inclusion relation can be shown by using 
\eqref{e3.7}.  By applying Lemma \ref{lem3.4} we have 
\begin{align*} 
P^{\ell}_h &\subset \cup\{v\tau s\Gamma : 1-c_4 B_1\delta \leq s\leq 
1+c_4 B_1\delta\} 
\\ 
&= \cup\{\tau s\Gamma : v(1-c_4 B_1\delta) \leq s\leq 
v(1+c_4 B_1\delta)\}  
\\ 
&\subset \cup\{\tau s\Gamma : (1-B_1\delta)(1-c_4 B\delta) \leq s\leq 
(1+B_1\delta)(1+c_4 B_1\delta)\}  
\\ 
&\subset \cup\{\tau s\Gamma : 1-D\delta \leq s\leq 1+D\delta\}  
\end{align*}
for some positive constant $D$.  This completes the proof. 
\end{proof} 

\begin{proof}[Proof of Lemma $\ref{lem3.3}$]  
We may assume that $P^{\ell}_h \cap 
(\cup_{1\leq t\leq 2}\,t^{-1}\supp(\phi^{(\delta)}))\neq \emptyset$.  Then we 
can find $\tau\in [1/2, 1]$ such that $P^{\ell}_h \cap 
(\tau\supp(\phi^{(\delta)}))\neq \emptyset$, which implies by Lemma \ref{lem3.6} that  for $t\in [2^n, 2^{n+1}]$  
\begin{equation*} 
t2^{-n}P^{\ell}_h  \subset \cup\{st2^{-n}\tau\Gamma: 
1-D\delta\leq s\leq 1+D\delta\}.   
\end{equation*} 
So,  if $t\in E(n,\ell,h)$, by Lemmas \ref{lem3.4} and \ref{lem3.5}, 
\begin{equation*} 
\frac{1-B_1\delta}{2^{-n}\tau(1+D\delta)}\leq t \leq 
\frac{1+B_1\delta}{2^{-n}\tau(1-D\delta)}. 
\end{equation*} 
It follows that 
\begin{equation*} 
\int_0^\infty  \chi_{E(n,\ell,h)}(t) \, \frac{dt}{t} \leq 
\int_{\frac{1-B_1\delta}{2^{-n}\tau(1+D\delta)}}^
{\frac{1+B_1\delta}{2^{-n}\tau(1-D\delta)}} \, \frac{dt}{t}
\leq \log \frac{1+B_1\delta}{1-B_1\delta} + \log \frac{1+D\delta}{1-D\delta}
\leq C\delta. 
\end{equation*}  
\end{proof} 
\par   
Let $t\in [2^n,2^{n+1}]$. Then $2^{-n}t\in [1,2]$ and by Lemma \ref{lem3.2} we 
have  
\begin{align*} 
s^{(\delta,t)}_\ell(\xi)&=s^{(\delta,2^{-n}t)}_\ell(2^n\xi) 
\\ 
&=\sum_{{(k,j) \in \mathscr G_{\ell, 2^{-n}t}} }
\sum_{h\in \mathscr F_{\ell, 2^{-n}t}} 
\chi_{P^{\ell,k,j}_h}(2^n\xi) s^{(\delta,2^{-n} t)}_\ell(2^n\xi)  
\\ 
&=\sum_{{(k,j) \in \mathscr G_{\ell, 2^{-n}t}} }
\sum_{h\in \mathscr F_{\ell, 2^{-n}t}} 
\chi_{2^{-n}P^{\ell,k,j}_h}(\xi) s^{(\delta,t)}_\ell(\xi). 
\end{align*}    
Using this and applying Lemma \ref{lem3.3}, we see that 
\begin{align*} 
&V(f)^2= \sum_\ell \sum_{n\in \Bbb Z}\int_{2^n}^{2^{n+1}}
 |S^{(\delta,t)}_\ell f|^2\, \frac{dt}{t}
\\ 
&\leq C\sum_\ell \sum_{n\in \Bbb Z}\int_{2^n}^{2^{n+1}} 
\sum_{{(k,j) \in \mathscr G_{\ell, 2^{-n}t}} }
\sum_{h\in \mathscr F_{\ell, 2^{-n}t}} 
\left|T_{2^{-n}P^{\ell,k,j}_h} 
S^{(\delta,t)}_\ell f\right|^2 \, \frac{dt}{t}
\\ 
&\leq C \sum_{n}\sum_{k}\sum_{j}\sum_\ell \sum_{h}
\int \chi_{_{E(n,\ell,h)}}(t)\left|T_{2^{-n}P^{\ell,k,j}_h} 
S^{(\delta,t)}_\ell f\right|^2 \, \frac{dt}{t}
\\ 
&\leq C\sum_{n, k, j, \ell, h}
\int \chi_{_{E(n,\ell,h)}}(t) \sup_{t\in [2^n,2^{n+1}]} 
\left|S^{(\delta,t)}_\ell T_{2^{-n}P^{\ell,k,j}_h} f\right|^2 
\, \frac{dt}{t}
\\ 
&\leq C\delta \sum_{n, k, j, \ell, h}
 \sup_{t\in [2^n,2^{n+1}]} 
\left|S^{(\delta,t)}_\ell T_{2^{-n}P^{\ell,k,j}_h} f\right|^2.  
\end{align*}  
Thus  we have 
\begin{equation}\label{e3.8} 
V(f)(x) 
\leq C\delta^{1/2}
\left(\sum_{n, k, j, \ell, h} \sup_{t\in [2^n,2^{n+1}]} 
\left|S^{(\delta,t)}_\ell T_{2^{-n}P^{\ell,k,j}_h} f\right|^2\right)^{1/2}. 
\end{equation}
\par 
Observe that $S^{(\delta,t)}_\ell(f)=K_{\ell,t}*f$ with 
\begin{equation}\label{e3.9} 
K_{\ell,t}(x)=\int_{\Bbb R^2} \zeta_\ell(t\xi_1)\phi(t\xi)
e^{2\pi i\langle x,\xi\rangle} \, d\xi.  
\end{equation} 
We can see that  
\begin{equation}\label{e3.10}
 \left|(\partial/\partial t_\ell)^r(\partial/\partial n_\ell)^s 
\zeta_\ell(t\xi_1)\phi(t\xi) \right|\leq C_{r,s}
(t^{-1}\delta^{1/2})^{-r}(t^{-1}\delta)^{-s}, 
\end{equation} 
where $\partial/\partial t_\ell$ and $\partial/\partial n_\ell$ denote 
differentiations in the directions $t_\ell$ and $n_\ell$, where 
$t_\ell=(t_\ell^{(1)}, t_\ell^{(2)})=
(1,\psi'(a_{\ell-1}))'$ ($\xi'=|\xi|^{-1}\xi$), 
$n_\ell=(-t_\ell^{(2)},t_\ell^{(1)})$,   
and we recall that 
$\phi(\xi)=\Phi(\delta^{-1}(\xi_2-\psi(\xi_1)))b(\xi)$. 
The estimates \eqref{e3.10} follow by the 
observations as in \cite[p. 310]{Sa}.  
Let $O_\ell\in O(2)$ be such that $O_\ell e_1=t_\ell$ and 
$O_\ell e_2=n_\ell$ with $e_1=(1,0)$, $e_2=(0,1)$. 
Then applying integration by parts in \eqref{e3.9} and using \eqref{e3.10}, 
we see that 
\begin{equation}\label{e3.11} 
|K_{\ell,t}(O_\ell x)|
\leq C_{\alpha,\beta}t^{-2}\delta^{3/2}|t^{-1}\delta^{1/2}x_1|^{-\alpha} 
|t^{-1}\delta x_2|^{-\beta}
\end{equation} 
for $\alpha, \beta \in \Bbb Z\cap [0,3]$.  Taking 
$(\alpha,\beta)=(3,0), (0,3)$ in \eqref{e3.11}, we have 
\begin{equation*} 
\sup_{t\in [2^n,2^{n+1}]}|S^{(\delta,t)}_{\ell}f|\leq C\sum_{\nu=0}^\infty 
2^{-\nu} |E_{\ell,n,\nu}|^{-1}\chi_{E_{\ell,n,\nu}}*|f|, 
\end{equation*} 
where 
\begin{equation*} 
 E_{\ell,n,\nu}=\{O_\ell x: |x_1|\leq 2^\nu 2^{n}\delta^{-1/2}, 
|x_2|\leq 2^\nu 2^{n}\delta^{-1}\}. 
\end{equation*} 
Therefore by \eqref{e3.8} we see that 
\begin{multline}\label{e3.12} 
\|V(f)\|_4
\\ 
\leq C\delta^{1/2}\sum_{\nu=0}^\infty 
2^{-\nu}\left\|\left(\sum_{n,k,j,\ell,h}\left||E_{\ell,n,\nu}|^{-1}
\chi_{E_{\ell,n,\nu}}*\left|T_{2^{-n}P^{\ell,k,j}_h} f\right|
\right|^2\right)^{1/2} \right\|_4.  
\end{multline} 
\par  
Applying Lemmas \ref{lem2.1} and \ref{lem2.2}, we have the following result 
which will be used in estimating 
the $L^4$ norms on the right hand side of \eqref{e3.12}.  
\begin{proposition}\label{prop3.7} 
There exists a constant $C>0$ such that 
\begin{equation*} 
\left\|\left(\sum_{n,k,j,\ell,h}\left|T_{2^{-n}P^{\ell,k,j}_h} f\right|^2\right)^{1/2}\right\|_4 \leq 
C\left(\log\frac{1}{\delta}\right)^{(\beta+3\alpha)/2} \|f\|_4, \quad 
\beta=3\Theta+2\beta_0,     
\end{equation*} 
where $\Theta$, $\beta_0$ and $\alpha$ are as in \eqref{e3.3}. 
\end{proposition} 
To prove this we also need the following. 
\begin{lemma} \label{lem3.8}
Let 
$\omega_k\times H_j\subset 
[-b,b]\times [c,d]$. Then there exists a constant $C$ 
independent of $k, j$ such that 
\begin{equation*} 
\card\{\ell: \widetilde{\Delta}_\ell\cap (\omega_k\times H_j)\neq \emptyset\}
\leq C.  
\end{equation*} 
\end{lemma} 
\begin{proof} 
Let $\Delta_\ell\cap (\omega_k\times H_j)\neq \emptyset$. 
Then 
$L_{\alpha_\ell}\cap (\omega_k\times H_j)\neq \emptyset$ 
for some $\alpha_\ell\in [a_{\ell-1}, a_\ell]$, where 
$L_{\alpha_\ell}=\{\xi: \xi_1=\Psi(\alpha_\ell) \xi_2\}$, 
$\Psi(t)=t/\psi(t)$.  
\par 
Suppose that $a_{k-1}\geq 0$. If $\omega_k=[a_{k-1},a_k]$,  
$H_j=[b_{j-1}, b_j]$, then 
\begin{equation*} 
a_{k-1}/b_{j}\leq \Psi(\alpha_\ell)\leq a_k/b_{j-1}.   
\end{equation*} 
If  also $\Delta_m\cap (\omega_k\times H_j)\neq \emptyset$,  
then we take $\alpha_m\in [a_{m-1}, a_m]$ such that 
\begin{equation*} 
a_{k-1}/b_{j}\leq \Psi(\alpha_m)\leq a_k/b_{j-1}.   
\end{equation*} 
Thus 
\begin{equation} \label{e3.13}
\left|\Psi(\alpha_\ell)-\Psi(\alpha_m) \right|
\leq a_k/b_{j-1}-a_{k-1}/b_{j}=
\frac{\delta^{1/2}(b_j+a_{k-1})}{b_jb_{j-1}}. 
\end{equation} 
Since $\psi'(t)\neq \psi(t)/t$ on $I$, we see that $|\Psi'(t)|\geq c$ 
on $I$ with a constant $c>0$. Thus by \eqref{e3.13} 
we have 
\begin{equation*} 
c|\alpha_\ell-\alpha_m|\leq C\delta^{1/2},  
\end{equation*} 
which implies that $|\ell-m|\leq C$.  From this we can derive what is claimed.  \par 
Suppose that $a_k\leq 0$.  Then 
\begin{equation*} 
a_{k-1}/b_{j-1}\leq \Psi(\alpha_\ell)\leq a_k/b_{j}.   
\end{equation*} 
Since 
\begin{equation*} \label{}
a_k/b_{j}-a_{k-1}/b_{j-1}=
\frac{\delta^{1/2}(b_{j-1}-a_{k-1})}{b_jb_{j-1}},  
\end{equation*} 
 arguing as above, we can handle this case. 
\par 
Suppose that $a_{k-1}<0<a_k$.  Then 
\begin{equation*} 
a_{k-1}/b_{j-1}\leq \Psi(\alpha_\ell)\leq a_k/b_{j-1}.   
\end{equation*} 
 Using this we can also get the desired result.   
\end{proof}  

\begin{proof}[Proof of Proposition $\ref{prop3.7}$] 
Let $\Omega_{N_*}=\{(1,\psi'(a_k))', (a_m,\psi(a_m))': 0\leq k, m\leq N\}$, 
$N_*=2(N+1)$. 
Let $w$ be a bounded, non-negative function with compact support. Then 
by Lemma \ref{lem2.2} and Lemma \ref{lem3.8} we see that 
\begin{align*} \label{}
&\int \sum_{n,k,j,\ell,h}\left|T_{2^{-n}P^{\ell,k,j}_h} f\right|^2
w\, dx 
\\ 
&\leq C\left(\frac{1}{s-1}\right)^{\Theta} 
\sum_{n,k,j,\ell}\int \left|T_{2^{-n}\Delta_\ell}
T_{2^{-n}(\omega_k\times H_j)}
T_{2^{-n}(\Bbb R\times [c,d])} f\right|^2
(M_{\Omega_{N_*}}(w^s))^{1/s}\, dx 
\\ 
&\leq C\left(\frac{1}{s-1}\right)^{\Theta+2\beta_0} 
\sum_{n,k,j}\int \left|T_{2^{-n}(\omega_k\times H_j)}
T_{2^{-n}(\Bbb R\times [c,d])} f\right|^2
(M_{\Omega_{N_*}}^3(w^s))^{1/s}\, dx 
\\ 
&\leq C\left(\frac{1}{s-1}\right)^{2\Theta+2\beta_0} 
\sum_{n,k}\int \left|T_{2^{-n}(\omega_k\times \Bbb R)}
T_{2^{-n}(\Bbb R\times [c,d])} f\right|^2
(M_1M_{\Omega_{N_*}}^3(w^s))^{1/s}\, dx 
\\ 
&\leq C\left(\frac{1}{s-1}\right)^{3\Theta+2\beta_0}
\sum_{n} \int \left|T_{2^{-n}(\Bbb R\times [c,d])} f\right|^2\, 
(M_1^2M_{\Omega_{N_*}}^3(w^s))^{1/s}\, dx, 
\end{align*} 
where the second inequality follows by Lemma \ref{lem3.8} and estimates shown 
similarly to \eqref{e2.2} by applying \eqref{e2.1}, and the other 
inequalities are derived from Lemma \ref{lem2.2}. 
\par 
Thus, applying the Schwarz inequality, we have 
\begin{align} \label{e3.14}
&\int \sum_{n,k,j,\ell,h}\left|T_{2^{-n}P^{\ell,k,j}_h} f\right|^2
w\, dx
\\ 
&\leq  C\left(\frac{1}{s-1}\right)^{3\Theta+2\beta_0} \left\|
\left(\sum_{n}\left|T_{2^{-n}(\Bbb R\times [c,d])} f\right|^2\right)^{1/2}
\right\|_4^2 \left\|(M_1^2M_{\Omega_{N_*}}^3(w^s))^{1/s}\right\|_2  \notag 
\\ 
&\leq  C\left(\frac{1}{s-1}\right)^{3\Theta+2\beta_0} \left\|f\right\|_4^2 
\left\|(M_1^2M_{\Omega_{N_*}}^3(w^s))^{1/s}\right\|_2,   \notag   
\end{align} 
where the last inequality follows by the Littlewood-Paley estimates. 
\par 
Let $\beta=3\Theta+2\beta_0$.  
We estimate $A_0:=\left(1/(s-1)\right)^{\beta}\left
\|(M_1^2M_{\Omega_{N_*}}^3(w^s))^{1/s}\right\|_2$ as follows. 
Let $s=1+(\log N)^{-1}$. Interpolating between the estimates 
$\|M_{\Omega_{N_*}}(f)\|_{4/3}\leq CN\|f\|_{4/3}$ and $\|M_{\Omega_{N_*}}(f)\|_2\leq C(\log N)^\alpha\|f\|_2$ (Lemma \ref{lem2.1}), 
we have $\|M_{\Omega_{N_*}}(f)\|_{2/s}\leq C_N\|f\|_{2/s}$, where 
$C_N\leq CN^{1-\theta}(\log N)^{\theta\alpha}$ with $1-\theta=2(\log N)^{-1}$, 
which implies that 
\begin{equation*} 
C_N\leq CN^{2(\log N)^{-1}}(\log N)^{\alpha(1-2(\log N)^{-1})}
\leq C(\log N)^\alpha. 
\end{equation*} 
Also, $(1/(s-1))^{\beta} =(\log N)^\beta$.  
Thus 
\begin{align*} \label{} 
A_0&\leq C(\log N)^\beta\left\|M_{\Omega_{N_*}}^3(w^s)\right\|_{2/s}^{1/s} 
\\ 
&\leq C(\log N)^\beta(\log N)^{3\alpha/s}\|w\|_2 
\\ 
&\leq C(\log N)^{\beta+3\alpha}\|w\|_2.   
\end{align*} 
Taking the supremum in \eqref{e3.14} over $w$ with $\|w\|_2\leq 1$, we have the conclusion of 
Proposition $\ref{prop3.7}$. 
\end{proof} 
Take a non-negative $w\in C_0^\infty(\Bbb R^2)$. 
By Lemma \ref{lem2.1} and Proposition \ref{prop3.7} we see that 
\begin{align*} 
&\int \sum_{n,k,j,\ell,h}\left||E_{\ell,n,\nu}|^{-1}
\chi_{E_{\ell,n,\nu}}*\left|T_{2^{-n}P^{\ell,k,j}_h} f\right|(x)
\right|^2 w(x)\, dx 
\\  
&\leq   \int \sum_{n,k,j,\ell,h}|E_{\ell,n,\nu}|^{-1}
\chi_{E_{\ell,n,\nu}}*\left|T_{2^{-n}P^{\ell,k,j}_h} f\right|^2(x)
 w(x)\,dx 
\\
&= \int \sum_{n,k,j,\ell,h}
\left|T_{2^{-n}P^{\ell,k,j}_h} f(y)\right|^2|E_{\ell,n,\nu}|^{-1}
\chi_{E_{\ell,n,\nu}}* w(y)\,dy 
\\
&\leq \left\|\left(\sum_{n,k,j,\ell,h}
\left|T_{2^{-n}P^{\ell,k,j}_h} f\right|^2
\right)^{1/2} \right\|_4^2\|M_{\Omega_{N_*}}w\|_2 
\\ 
&\leq C\left(\log\frac{1}{\delta}\right)^{\beta+4\alpha}\|f\|_4^2\|w\|_2.   
\end{align*} 
Taking the supremum over $w$ with $\|w\|_2\leq 1$, we have 
\begin{equation*} 
\left\|\left(\sum_{n,k,j,\ell,h}\left||E_{\ell,n,\nu}|^{-1}
\chi_{E_{\ell,n,\nu}}*\left|T_{2^{-n}P^{\ell,k,j}_h} f
\right|\right|^2\right)^{1/2} 
\right\|_4 \leq C\left(\log\frac{1}{\delta}\right)^{(\beta/2)+2\alpha}\|f\|_4. 
\end{equation*} 
Using this in \eqref{e3.12}, we see that 
\begin{equation*} 
\|V(f)\|_4 \leq 
C\delta^{1/2}\left(\log\frac{1}{\delta}\right)^{(\beta/2)+2\alpha}\|f\|_4. 
\end{equation*}  
By \eqref{e3.6} this proves \eqref{e3.3} for $k=0$ and $i=1$. 
As we have already seen, 
this will lead to the estimate \eqref{e3.4}.  
This completes the proof of Theorem \ref{thm1.2} under $(B.1)$ with $\Psi'>0$. 

\begin{proof}[Proof of Theorem $\ref{thm1.3}$ in the
 case $(B.1)$ with $\Psi'>0$]   
Let $\phi_0 \in C_0^\infty(\Bbb R)$ be supported in $[1/2, 2]$ and 
$\sum_{n=0}^\infty \phi_0(2^n t)=1$ for $0<t<1$. 
Decompose 
\begin{align*} 
\Psi_{\lambda,\theta}(\xi)
&=b(\xi)\varphi_{\lambda,\theta}(\xi_2-\psi(\xi_1)) =r(\xi)+
\sum_{n=1}^\infty b(\xi)\varphi_{\lambda,\theta}(\xi_2-\psi(\xi_1)) 
\phi_0(2^n(\xi_2-\psi(\xi_1)))
\\ 
&= r(\xi)+ \sum_{n=1}^\infty  2^{-n\lambda}n^{-\theta} 
b(\xi)\left(2^{n\lambda}n^{\theta}\varphi_{\lambda,\theta}(\xi_2-\psi(\xi_1))
\right) \phi_0(2^n(\xi_2-\psi(\xi_1))) 
\\ 
&= r(\xi)+ \sum_{n=1}^\infty  2^{-n\lambda}n^{-\theta} 
b(\xi)\varphi_{n,\lambda,\theta}(2^n(\xi_2-\psi(\xi_1))), 
\end{align*} 
where $r\in C_0^\infty(\Bbb R^2)$, $0\notin \supp(r)$, 
and 
$\varphi_{n,\lambda,\theta}(u)=u^\lambda n^\theta 
(\log(2+2^n/u))^{-\theta}\phi_0(u)$. 
Put 
$\Psi_{n, \lambda, \theta}(\xi)
= b(\xi)\varphi_{n,\lambda,\theta}(2^n(\xi_2-\psi(\xi_1)))$. 
Applying Theorem \ref{thm1.2} and a triangle inequality, if $\lambda$ and 
$\theta$ satisfy the conditions of Theorem \ref{thm1.3}, 
we see that 
\begin{align*} 
\|g_{\mathscr{F}^{-1}(\Psi_{\lambda,\theta})}(f)\|_4 &\leq 
\|g_{\mathscr{F}^{-1}(r)}(f)\|_4 +
\sum_{n=1}^\infty 2^{-n\lambda}n^{-\theta}
\|g_{\mathscr{F}^{-1}(\Psi_{n,\lambda,\theta})}(f)\|_4 
\\ 
&\leq C\|f\|_4 +C\sum_{n=1}^\infty 2^{-n\lambda}n^{-\theta} n^\tau 2^{-n/2}
\|f\|_4
\\ 
&\leq C\|f\|_4.  
\end{align*} 
This completes the proof of Theorem \ref{thm1.3} for the case $(B.1)$ with 
$\Psi'>0$.  
\end{proof}

\section{ Proof of Theorem \ref{thm1.1} for case $(B.1)$ with  $\Psi'>0$}
\label{sec4}  
In this section we prove Theorem \ref{thm1.1} 
under the conditions that 
$\Gamma \subset (-b,b)\times (c,d)$, $0<b$, $0<c<d$, 
$I=[A,B]\subset (-b,b)$ and $\Psi'>0$.  
\par 
We first prepare some results involving homogeneous functions of 
degree one in a cone, which will be used in what follows.  
Let $H$ be an interval in $\Bbb R$ with $|H|<\pi$. Define 
a cone $C_H$ by 
\begin{equation*} 
C_H=\{r(\cos\theta, \sin\theta): \theta\in H, \, r>0\}. 
\end{equation*}    
For an appropriate function $f$, let 
\begin{align*} 
B^\lambda_R f(x)&=\int_{\rho(\xi)<R}\hat{f}(\xi)
b(R^{-1}\xi)(1-R^{-1}\rho(\xi))_+^\lambda 
e^{2\pi i\langle x, \xi\rangle}\, d\xi
\\ 
&= R^{-\lambda}
\int_{\rho(\xi)<R}\hat{f}(\xi)b(R^{-1}\xi)(R-\rho(\xi))_+^\lambda 
e^{2\pi i\langle x, \xi\rangle}\, d\xi,     
\end{align*} 
where  $\rho$ is a non-negative function 
 on $\Bbb R^2$ and $b \in C_0^\infty(\Bbb R^2)$ such that 
$0 \not\in \supp(b) \subset C_H$ with an open interval $H$, 
 $\supp(b)\subset \{0< r_1<|\xi|< r_2\}$.  
 We assume that $\rho(t\xi)=t\rho(\xi)$ if $\xi \in  C_H$ and $ t>0$.  
Then we have the following  
(see \cite[\S 5, Chap.\,VII]{SW} and also \cite[Lemma 4]{Sa-0} for relevant 
results). 

\begin{lemma}\label{lem4.1} 
If $\beta>1/2$ and $\delta>-1/2$, then 
\begin{equation*} 
\left|B^{\delta+\beta}_R f(x)\right|\leq C_{\delta,\beta}
\left(\int_0^1(1-s)^{2(\beta-1)}s^{2\delta}\, ds\right)^{1/2}
\left(R^{-1}\int_0^R\left| (\mathscr F^{-1}b)_{R^{-1}}*\widetilde{B}_s^\delta 
f(x)\right|^2\, ds\right)^{1/2}.   
\end{equation*} 
Here  $C_{\delta,\beta}
=\Gamma(\delta+\beta+1)/(\Gamma(\delta+1)\Gamma(\beta))$ and 
\begin{align*} 
\widetilde{B}^\delta_s f(x)=\int_{\rho(\xi)<s}\hat{f}(\xi)
\widetilde{b}(s^{-1}\xi)(1-s^{-1}\rho(\xi))_+^\delta  
e^{2\pi i\langle x, \xi\rangle}\, d\xi, 
\end{align*} 
where $\widetilde{b}$ is a function in $C^\infty(\Bbb R^2)$ such that $0\not\in \supp(\widetilde{b}) \subset C_{H}$ 
and such that there exists a compact subinterval $H'$ of $H$ for which 
we have $\supp(b) \subset C_{H'}$ and $\widetilde{b}(\xi)=1$ for all 
$\xi \in C_{H'}$ with $|\xi|\geq r_1$.   
\end{lemma}

\begin{proof} 
Using the formula 
\begin{equation*} 
(R-\rho(\xi))_+^{\delta+\beta}= C_{\delta,\beta}\int_0^{(R-\rho(\xi))_+}
((R-\rho(\xi))_+-u)^{\beta-1}u^\delta \, du, 
\end{equation*} 
we have 
\begin{multline*} 
B^{\delta+\beta}_R f(x) 
\\ 
= C_{\delta,\beta}R^{-\delta-\beta}
\int_{\rho(\xi)<R}
\int_0^{(R-\rho(\xi))_+}
((R-\rho(\xi))_+-u)^{\beta-1}u^\delta \, du\, 
b(R^{-1}\xi)\hat{f}(\xi) 
e^{2\pi i\langle x, \xi\rangle}\, d\xi.      
\end{multline*} 
Changing variables $u=s-\rho(\xi)$, we see that this is equal to 
\begin{align*} 
&C_{\delta,\beta}R^{-\delta-\beta}
\int_{\rho(\xi)<R}
\int_{\rho(\xi)}^{R}
(R-s)^{\beta-1}(s-\rho(\xi))^\delta \, ds\, 
b(R^{-1}\xi)\hat{f}(\xi) e^{2\pi i\langle x, \xi\rangle}\, d\xi 
\\ 
&= C_{\delta,\beta}R^{-\delta-\beta}
\int_{0}^{R}
\int_{\rho(\xi)<s}
(s-\rho(\xi))^\delta 
b(R^{-1}\xi)\hat{f}(\xi) e^{2\pi i\langle x, \xi\rangle}\, d\xi\,  
(R-s)^{\beta-1}\, ds 
\\ 
&=C_{\delta,\beta}R^{-\delta-\beta} 
\\ 
&\quad\times 
\int_{0}^{R}(R-s)^{\beta-1}s^\delta 
\int_{\rho(\xi)<s}
(1-\rho(s^{-1}\xi))^\delta 
\widetilde{b}(s^{-1}\xi)b(R^{-1}\xi)\hat{f}(\xi) 
e^{2\pi i\langle x, \xi\rangle}\, d\xi  \, ds 
\\ 
&=C_{\delta,\beta}R^{-\delta-\beta}
\int_{0}^{R}(R-s)^{\beta-1}s^\delta (\mathscr F^{-1} b)_{R^{-1}}*
\widetilde{B}^\delta_s f(x)  \, ds.   
\end{align*} 
Thus by applying the Schwarz inequality, we have 
\begin{align*} 
&\left|B^{\delta+\beta}_R f(x)\right|
\\ 
&\leq  
C_{\delta,\beta}R^{-\delta-\beta}
\left(\int_{0}^{R}(R-s)^{2(\beta-1)}s^{2\delta} \, ds\right)^{1/2}  
\left(\int_{0}^{R}\left|(\mathscr F^{-1} b)_{R^{-1}}*
\widetilde{B}^\delta_s f(x)\right|^2  \, ds\right)^{1/2}
\\ 
&=C_{\delta,\beta}R^{-1/2}
\left(\int_{0}^{1}(1-s)^{2(\beta-1)}s^{2\delta} \, ds\right)^{1/2}  
\left(\int_{0}^{R}\left|(\mathscr F^{-1} b)_{R^{-1}}*\widetilde{B}^\delta_s 
f(x)\right|^2  \, ds\right)^{1/2}. 
\end{align*} 
This completes the proof of Lemma \ref{lem4.1}. 
\end{proof}

\begin{corollary} \label{cor4.2} 
Let $\beta, \delta$, $\rho$, $b$, $\widetilde{b}$, $B_R^\lambda$ and 
$\widetilde{B}_R^\lambda$
 be as in Lemma $\ref{lem4.1}$.  Let $B_*^\lambda f(x)
=\sup_{R>0} |B_R^\lambda f(x)|$ and 
$\psi^\delta(\xi)=\widetilde{b}(\xi)(1-\rho(\xi))_+^\delta$.  Then 
\begin{equation*} 
B_*^{\delta+\beta}f(x)\leq C M(g_{\mathscr{F}^{-1}(\psi^\delta)}f)(x), 
\end{equation*} 
where $M$ denotes the Hardy-Littlewood maximal operator. 
\end{corollary} 
\begin{proof} 
 By Lemma \ref{lem4.1} and Minkowski's inequality we have 
\begin{multline*}
\left|B^{\delta+\beta}_R f(x)\right|
\\ 
\leq C_{\delta,\beta}
\left(\int_0^1(1-s)^{2(\beta-1)}s^{2\delta}\, ds\right)^{1/2}
\left|(\mathscr F^{-1}b)_{R^{-1}}\right|*\left(R^{-1}\int_0^R
\left|\widetilde{B}_s^\delta f(\cdot)\right|^2\, ds\right)^{1/2}(x).   
\end{multline*} 
Thus 
\begin{equation*}
B^{\delta+\beta}_* f(x)\leq C 
\sup_{R>0}\left(\left|(\mathscr F^{-1}b)_{R^{-1}}
\right|*g_{\mathscr{F}^{-1}(\psi^\delta)}f(x)\right) 
\leq CM(g_{\mathscr{F}^{-1}(\psi^\delta)} f)(x).  
\end{equation*} 
\end{proof} 
\par 
We also consider 
\begin{align*} 
C^\lambda_R f(x)&=\int_{\rho(\xi)>R}\hat{f}(\xi)
b(R^{-1}\xi)(R^{-1}\rho(\xi)-1)_+^\lambda 
e^{2\pi i\langle x, \xi\rangle}\, d\xi
\\ 
&= R^{-\lambda}
\int_{\rho(\xi)>R}\hat{f}(\xi)b(R^{-1}\xi)(\rho(\xi)-R)_+^\lambda 
e^{2\pi i\langle x, \xi\rangle}\, d\xi,     
\end{align*} 
where $\rho$ and $b \in C_0^\infty(\Bbb R^2)$ 
are as in the definition of $B^\lambda_R$. 
Then we have the following. 

\begin{lemma}\label{lem4.3} 
Let $\beta>1/2$, $\delta>-1/2$. 
Suppose that $c_0=\sup_{\xi\in C_H}\rho(\xi')<\infty$, 
with $\xi'=\xi/|\xi|$. Put  $d_0=c_0r_2$.  Then, 
$C^{\delta+\beta}_R f=0$ if $d_0<1$ and if $d_0\geq 1$ we have 
\begin{equation*} 
\left|C^{\delta+\beta}_R f(x)\right|\leq C_{\delta,\beta}
\left(\int_1^{d_0}(s-1)^{2(\beta-1)}s^{2\delta}\, ds\right)^{1/ 2}
\left(R^{-1}\int_R^{d_0R}\left| (\mathscr F^{-1}b)_{R^{-1}}*\widetilde{C}_s^\delta 
f(x)\right|^2\, ds\right)^{1/2},  
\end{equation*} 
where  the constant $C_{\delta,\beta}$ is as in Lemma $\ref{lem4.1}$ and 
\begin{align*} 
\widetilde{C}^\delta_s f(x)=\int_{\rho(\xi)>s}\hat{f}(\xi)
\widetilde{b}(s^{-1}\xi)(s^{-1}\rho(\xi)-1)_+^\delta  
e^{2\pi i\langle x, \xi\rangle}\, d\xi, 
\end{align*} 
where $\widetilde{b}$ is a function in $C^\infty_0(\Bbb R^2)$ such that 
$0\not\in \supp(\widetilde{b}) \subset C_{H}$ 
and such that 
there exists a compact subinterval $H'$ of $H$ for which 
we have $\supp(b) \subset C_{H'}$ and $\widetilde{b}(\xi)=1$ if 
 $d_0^{-1}r_1 \leq |\xi|\leq r_2$ and $\xi \in C_{H'}$. 
\end{lemma}

\begin{proof} 
As in the case of $B^\lambda_R$, using the formula 
\begin{equation*} 
(\rho(\xi)-R)_+^{\delta+\beta}= C_{\delta,\beta}\int_0^{(\rho(\xi)-R)_+}
((\rho(\xi)-R)_+-u)^{\beta-1}u^\delta \, du, 
\end{equation*} 
we have 
\begin{multline*} 
C^{\delta+\beta}_R f(x) 
\\ 
= C_{\delta,\beta}R^{-\delta-\beta}
\int_{\rho(\xi)>R}
\int_0^{(\rho(\xi)-R)_+}
((\rho(\xi)-R)_+-u)^{\beta-1}u^\delta \, du\, 
b(R^{-1}\xi)\hat{f}(\xi) 
e^{2\pi i\langle x, \xi\rangle}\, d\xi.      
\end{multline*} 
Changing variables $u=\rho(\xi)-s$, we see that this is equal to 
\begin{align*} 
&C_{\delta,\beta}R^{-\delta-\beta}
\int_{\rho(\xi)>R}
\int^{\rho(\xi)}_{R}
(s-R)^{\beta-1}(\rho(\xi)-s)^\delta \, ds\, 
b(R^{-1}\xi)\hat{f}(\xi) e^{2\pi i\langle x, \xi\rangle}\, d\xi 
\\ 
&= C_{\delta,\beta}R^{-\delta-\beta}
\int_{R}^{d_0R}
\int_{\rho(\xi)>s}
(\rho(\xi)-s)^\delta 
b(R^{-1}\xi)\hat{f}(\xi) e^{2\pi i\langle x, \xi\rangle}\, d\xi\,  
(s-R)^{\beta-1}\, ds 
\\ 
&=C_{\delta,\beta}R^{-\delta-\beta} 
\\ 
&\quad\times 
\int_{R}^{d_0 R}(s-R)^{\beta-1}s^\delta 
\int_{\rho(\xi)>s}
(\rho(s^{-1}\xi)-1)^\delta 
\widetilde{b}(s^{-1}\xi)b(R^{-1}\xi)\hat{f}(\xi) 
e^{2\pi i\langle x, \xi\rangle}\, d\xi  \, ds 
\\ 
&=C_{\delta,\beta}R^{-\delta-\beta}
\int_{R}^{d_0 R}(s-R)^{\beta-1}s^\delta (\mathscr F^{-1} b)_{R^{-1}}*
\widetilde{C}^\delta_s f(x)  \, ds.   
\end{align*} 
By the Schwarz inequality, it follows that  
\begin{align*} 
&\left|C^{\delta+\beta}_R f(x)\right|
\\ 
&\leq  
C_{\delta,\beta}R^{-\delta-\beta}
\left(\int_{R}^{d_0R}(s-R)^{2(\beta-1)}s^{2\delta} \, ds\right)^{1/2}  
\left(\int_{R}^{d_0R}\left|(\mathscr F^{-1} b)_{R^{-1}}*
\widetilde{C}^\delta_s f(x)\right|^2  \, ds\right)^{1/2}
\\ 
&=C_{\delta,\beta}R^{-1/2}
\left(\int_{1}^{d_0}(s-1)^{2(\beta-1)}s^{2\delta} \, ds\right)^{1/2}  
\left(\int_{R}^{d_0R}\left|(\mathscr F^{-1} b)_{R^{-1}}*\widetilde{C}^\delta_s 
f(x)\right|^2  \, ds\right)^{1/2}. 
\end{align*} 
This completes the proof of Lemma \ref{lem4.3}. 
\end{proof}

\begin{corollary} \label{cor4.4} 
Let $\beta, \delta$, $\rho$, $b$, $\widetilde{b}$, $C_R^\lambda$ and 
$\widetilde{C}_R^\lambda$
 be as in Lemma $\ref{lem4.3}$.  Let $C_*^\lambda f(x)
=\sup_{R>0} |C_R^\lambda f(x)|$. Put  
$\varphi^\delta(\xi)=\widetilde{b}(\xi)(\rho(\xi)-1)_+^\delta$.  Then 
\begin{equation*} 
C_*^{\delta+\beta}f(x)\leq C M(g_{\mathscr{F}^{-1}(\varphi^\delta)}f)(x), 
\end{equation*} 
where $M$ denotes the Hardy-Littlewood maximal operator as above. 
\end{corollary} 
\begin{proof} 
 As in the proof of Corollary \ref{cor4.2}, by Lemma \ref{lem4.3} and 
Minkowski's inequality we see that 
\begin{multline*}
\left|C^{\delta+\beta}_R f(x)\right|
\\ 
\leq C_{\delta,\beta}
\left(\int_1^{d_0}(s-1)^{2(\beta-1)}s^{2\delta}\, ds\right)^{1/2}
\left|(\mathscr F^{-1}b)_{R^{-1}}\right|*\left(R^{-1}\int_R^{d_0R}
\left|\widetilde{C}_s^\delta f(\cdot)\right|^2\, ds\right)^{1/2}(x).   
\end{multline*} 
Thus 
\begin{equation*}
C^{\delta+\beta}_* f(x)\leq C_{d_0} 
\sup_{R>0}\left(\left|(\mathscr F^{-1}b)_{R^{-1}}
\right|*g_{\mathscr{F}^{-1}(\varphi^\delta)}f(x)\right) 
\leq CM(g_{\mathscr{F}^{-1}(\varphi^\delta)} f)(x).  
\end{equation*} 
\end{proof} 

\par 
Now we can start the proof of the theorem in the situation stated. 
For $a_1<a_2$ let 
\begin{equation*} 
C(a_1,a_2)=\{\xi\in \Bbb R^2\setminus \{0\}: a_1\xi_2<\xi_1<a_2\xi_2,
\, \xi_2>0\}. 
\end{equation*}  
Take intervals $I_1=[\sigma, \tau]$, $I_2=[\sigma', \tau']$, 
$I_3=[\sigma'', \tau'']$, 
$-b<\sigma''<\sigma'<\sigma<\tau<\tau'<\tau''<b$ such that 
$\supp(\sigma_\lambda) \subset I_1\times \Bbb R$ and 
 $\Psi'>0$ on $I_3$. 
Let a function $\rho$ be defined on $C(a_1'',a_2'')$, 
$a_1''=\sigma''/\psi(\sigma'')$, $a_2''=\tau''/\psi(\tau'')$,   by 
\begin{equation*} 
\rho(u,v)=\frac{u}{\Psi^{-1}(\frac{u}{v})}, 
\end{equation*} 
if $u\neq 0$, where $\Psi: I_3 \to J:=\Psi(I_3)$, $\Psi^{-1}: J \to I_3$; 
also let $\rho(0,v)=v\Psi'(0)=v/\psi(0)$ if $0\in I_3$. 
We note that $\rho$ is non-negative and homogeneous of degree one on 
$C(a_1'',a_2'')$,  and 
$\rho(u, \psi(u))=1$ for $u\in I_3$.  We also note that 
$\Psi^{-1}(s)=(\Psi'(0))^{-1}s + O(s^2)$ near $s=0$ when $0\in J$. 
We can see that $\rho$ is infinitely differentiable in $C(a_1'',a_2'')$.  
Further, by taking account of a suitable partition of unity, 
we may assume that 
$\supp(a)\subset C(a_1,a_2)\cap (I_1\times \Bbb R)$, $a_1=\sigma/\psi(\sigma)$,  $a_2=\tau/\psi(\tau)$,   where $a$ is as in Theorem \ref{thm1.1} 
(see \eqref{e1.1}).  
\par 
 We have 
\begin{equation} \label{e4.1}
\widetilde{a}(\xi)(\rho(\xi)-1)_+^\lambda
=a(\xi)(\xi_2-\psi(\xi_1))_+^\lambda  
\end{equation}
for some $\widetilde{a}\in C_0^\infty(\Bbb R^2)$ with 
$\supp(a)=\supp(\widetilde{a})$.   
This can be seen as follows.  
Since $\partial_2 \rho>0$ on  $C(a_1'',a_2'')$, where 
$\partial_2=\partial/\partial v$, for $\xi\in \supp(a)$ 
we see  that   
\begin{align*} 
a(\xi)(\rho(\xi)-1)_+^\lambda &=
a(\xi)(\rho(\xi)-\rho(\xi_1,\psi(\xi_1)))_+^\lambda 
\\
&=a(\xi)\left(\int_0^1 \frac{\partial}{\partial t}
\rho(\xi_1,t(\xi_2-\psi(\xi_1))+\psi(\xi_1))\, dt\right)_+^\lambda  \notag 
\\ 
&= a(\xi)\left((\xi_2-\psi(\xi_1))\int_0^1 \partial_2 
\rho(\xi_1,t(\xi_2-\psi(\xi_1))+\psi(\xi_1))\, dt\right)_+^\lambda  \notag
\\ 
&= a(\xi)(\xi_2-\psi(\xi_1))_+^\lambda
\left(\int_0^1 \partial_2 
\rho(\xi_1,t(\xi_2-\psi(\xi_1))+\psi(\xi_1))\, dt\right)^\lambda  \notag
\\ 
&= a(\xi)g(\xi)(\xi_2-\psi(\xi_1))_+^\lambda     \notag
\end{align*} 
for some $g\in C_0^\infty(\Bbb R^2)$ such that $g \neq 0$ on $\supp(a)$. 
This implies \eqref{e4.1} by setting $\widetilde{a}=a/g$.     
\par 
Thus to prove Theorem \ref{thm1.1}  it suffices to show that  
$\|T_*^\lambda(f)\|_4 \leq C\|f\|_4$, where 
\begin{equation*} 
T_*^\lambda(f)(x)=\sup_{R>0}\left| 
\int \widetilde{a}(R^{-1}\xi)(\rho(R^{-1}\xi)-1)_+^\lambda  
\hat{f}(\xi)e^{2\pi i\langle x,\xi\rangle}\, d\xi \right|.  
\end{equation*} 
We have 
 $T_*^\lambda(f)\leq CM(g_{\mathscr{F}^{-1}(\varphi^\delta)}(f))$ 
with some $\delta>-1/2$ by Corollary \ref{cor4.4}, where
$\varphi^\delta(\xi)=A(\xi)(\rho(\xi)-1)^\delta_+$ and $A$ is related to 
$\widetilde{a}$ as $\widetilde{b}$ is related to $b$  
in Corollary \ref{cor4.4}; also we can assume that 
$\supp (A) \subset C(a_1,a_2)$.  
Therefore, to prove Theorem \ref{thm1.1} 
it suffices to show that $\|g_{\mathscr{F}^{-1}(\varphi^\delta)}(f)\|_4 
\leq C\|f\|_4$ for $\delta>-1/2$. 
\par  
Let $b_1(\xi_1)\in C_0^\infty(\Bbb R)$ satisfy that $b_1=1$ on $I_2$ and 
$\supp(b_1)\subset I_3$. Decompose $A(\xi)=b_1(\xi_1)A(\xi)+b_2(\xi_1)A(\xi)=: 
A_1(\xi)+A_2(\xi)$, where $b_2=1-b_1$. We observe that the function 
$B_2(\xi):=A_2(\xi)(\rho(\xi)-1)_+^\delta$ belongs to $C_0^\infty(\Bbb R^2)$, 
since $|\rho(\xi)-1|\geq c>0$ on $\supp(B_2)$,  and 
vanishes near the origin. Thus $g_{\mathscr{F}^{-1}(B_2)}$ is bounded on 
$L^p$, $1<p<\infty$. 
Since $\supp (A_1) \subset C(a_1,a_2)\cap(I_3\times \Bbb R)$,  in the same way 
as in \eqref{e4.1} we see that 
\begin{equation*} \label{e4.2}
B_1(\xi):=A_1(\xi)(\rho(\xi)-1)_+^\delta
=\widetilde{A_1}(\xi)(\xi_2-\psi(\xi_1))_+^\delta, 
\end{equation*}
where $\widetilde{A_1}\in C_0^\infty(\Bbb R^2)$ with 
$\supp(A_1)=\supp(\widetilde{A}_1)$.  
Therefore, applying Theorem \ref{thm1.3} 
for the case $\lambda>-1/2$ and $\theta=0$, we have 
the boundedness of $g_{\mathscr{F}^{-1}(B_1)}$, which concludes  
the proof of Theorem \ref{thm1.1} in the case $(B.1)$ with $\Psi'>0$.

\section{ Proofs of Theorems \ref{thm1.1}, \ref{thm1.2}, \ref{thm1.3} 
in full generality}
\label{sec5}

First we prove Theorem \ref{thm1.2} under the general 
conditions stated in the theorem.  
We have already proved the theorem in the case $(B.1)$  
when $\Psi'>0$ on $I$.
\par 
Proof of Theorem \ref{thm1.2} for the case $(B.1)$ with $\Psi'<0$ on $I$. 
We argue similarly to the case $\Psi'>0$. 
Recall that 
$\omega_k=[a_{k-1}, a_k]$, $a_k-a_{k-1}=\delta^{1/2}$, 
 $\cup \omega_k=[-b,b]$,  $H_j=[b_{j-1},b_j]$, $b_j-b_{j-1}=\delta^{1/2}$, 
$\cup H_j=[c,d]$. 
  Define, for $\ell$ with $[a_{\ell-2},a_{\ell+1}]\subset I=[A,B]$,  
\begin{gather*}   
\Delta_\ell=\{(\xi_1,\xi_2): \Psi(a_{\ell})\xi_2\leq \xi_1 
\leq \Psi(a_{\ell-1})\xi_2, \, \xi_2>0\}, \quad 
\widetilde{\Delta}_\ell=\cup_{|\ell'-\ell|\leq 1} \Delta_{\ell'} 
\end{gather*}  
and  recall $S^\ell_h$. 
Also define  $P^{\ell}_h$, $P^{\ell,k,j}_h$ as in Section \ref{sec3} 
by using  $\Delta_\ell$. 
\par 
Decompose $\phi^{(\delta)}$ as in \eqref{e3.2}. 
Arguing as in Section \ref{sec3}, by considering $g_{\mathscr F_i}^{(k)}(f)$, 
$1\leq i\leq4$, $0\leq k\leq L-1$,  
we need to prove estimates \eqref{e3.3}, from which 
we can deduce the desired estimates for $g_\eta(f)$ 
 (see \eqref{e3.4}) by reasoning as in Section \ref{sec3}. 
\par 
We give more specific arguments in the following, focusing on 
the case $k=0$ and $i=1$ in \eqref{e3.3}.  
We can estimate $g_{\mathscr F_1}^{(0)}(f)$ as in Section \ref{sec3} 
and we have an analogue of \eqref{e3.6}: 
\begin{equation}\label{e5.1}
\|g_{\mathscr F_1}^{(0)}(f)\|_4\leq  C\|V(f)\|_4, 
\end{equation} 
where $V(f)$ is defined as in Section \ref{sec3}. 
Also, we have  analogues of Lemma \ref{lem3.2} and Lemma \ref{lem3.3} with 
similar proofs. 
\par   
Let $t\in [2^n,2^{n+1}]$. Then $2^{-n}t\in [1,2]$ and by the analogue of 
Lemma \ref{lem3.2} we have  
\begin{align*} 
s^{(\delta,t)}_\ell(\xi)
=\sum_{k, j,h} 
\chi_{2^{-n}P^{\ell,k,j}_h}(\xi) s^{(\delta,t)}_\ell(\xi)  
\end{align*}    
as in Section \ref{sec3}. 
Using this and applying an analogue of Lemma \ref{lem3.3}, 
as in \eqref{e3.8} we see that  
\begin{equation}\label{e5.2} 
V(f)(x) 
\leq C\delta^{1/2}
\left(\sum_{n,k,j,\ell,h}
 \sup_{t\in [2^n,2^{n+1}]} 
\left|S^{(\delta,t)}_\ell T_{2^{-n}P^{\ell,k,j}_h} f\right|^2\right)^{1/2}, 
\end{equation} 
where $S^{(\delta,t)}_\ell$ is defined in the same way as in 
Section \ref{sec3}. 
\par    
We also have the following result by applying Lemmas \ref{lem2.1} and 
\ref{lem2.2}.  
\begin{proposition}\label{prop5.1} 
Let $\Theta$ and $\beta_0$ be as in Lemma $\ref{lem2.2}$ and \eqref{e2.1}, 
respectively, and let $\alpha$ be as in Lemma $\ref{lem2.1}$. Then 
\begin{equation*} 
\left\|\left(\sum_{n,k,\ell,j,h}\left|T_{2^{-n}P^{\ell,k,j}_h} f\right|^2\right)^{1/2}\right\|_4 \leq 
C\left(\log\frac{1}{\delta}\right)^{(\beta+3\alpha)/2} \|f\|_4, \quad 
\beta=3\Theta+2\beta_0.       
\end{equation*} 
\end{proposition} 
The following result is also needed in proving Proposition \ref{prop5.1}. 
\begin{lemma} \label{lem5.2}
Let 
$\omega_k\times H_j\subset [-b,b]\times [c,d]$. Suppose that $\Psi'<0$ on $I$. 
Then there exists a constant $C$ 
independent of $k, j$ such that 
\begin{equation*} 
\card\{\ell: \widetilde{\Delta}_\ell\cap (\omega_k\times H_j)\neq \emptyset\}
\leq C.  
\end{equation*} 
\end{lemma} 
\begin{proof} 
The proof is similar to the one for Lemma \ref{lem3.8}. 
\end{proof}  
\par 
By \eqref{e5.2}, Lemma \ref{lem2.1} and Proposition \ref{prop5.1} we have 
\begin{equation*} 
\|V(f)\|_4 \leq 
C\delta^{1/2}\left(\log\frac{1}{\delta}\right)^{(\beta/2)+2\alpha}\|f\|_4. 
\end{equation*}  
By this and \eqref{e5.1}, we have an analogue of \eqref{e3.3} for $k=0$ and 
$i=1$. Thus we have  estimates for $g_\eta$ analogous to \eqref{e3.4} 
 under the conditions of $(B.1)$ and $\Psi'<0$. 
\par 
Proof of Theorem \ref{thm1.2} for the case $(B.2)$. 
We consider $-\psi$ in place of $\psi$ in the definition of $g_\eta$ and apply 
the result of 
case $(B.1)$. We can see that this proves the result in the case 
$(B.2)$ by changing variables $\xi_2\to -\xi_2$ in $\phi$.  
\par 
Proof of Theorem \ref{thm1.2} for the case $(B.3)$. 
Similarly to Section \ref{sec3}, we decompose $[a, b]$ and $[-d, d]$: 
$[a, b]=\cup \omega_k$, $\omega_k=[a_{k-1}, a_k]$, $|\omega_k|=\delta^{1/2}$; 
$[-d, d]=\cup H_j$, $H_j=[b_{j-1},b_j]$, $|H_j|=\delta^{1/2}$.  
Recall that $\Psi_*(t)=\psi(t)/t$. Then we also have $\Psi_*'\neq 0$ on $I$ by 
$(A.2)$.   
We use $\Psi_*$ in the definition of $\Delta_\ell$ in place of $\Psi$: 
\begin{equation*}
\Delta_\ell=\{(\xi_1,\xi_2): \Psi_*(a_{\ell-1})\xi_1\leq \xi_2 
\leq \Psi_*(a_{\ell})\xi_1, \, \xi_1>0\}
\end{equation*} 
if $\Psi_*'>0$ on $I$ and 
\begin{equation*}
\Delta_\ell=\{(\xi_1,\xi_2): \Psi_*(a_{\ell})\xi_1\leq \xi_2 
\leq \Psi_*(a_{\ell-1})\xi_1, \, \xi_1>0\}
\end{equation*} 
if $\Psi_*'<0$ on $I$. 
Similarly, we also define $\widetilde{\Delta}_\ell$ and 
\begin{gather*} 
S^\ell_h=\{(\xi_1,\xi_2): \psi'(a_{\ell-1})\xi_1 +h\delta \leq \xi_2 
\leq \psi'(a_{\ell-1})\xi_1 +(h+1)\delta \}, \quad h\in \Bbb Z, 
\\ 
P^{\ell}_h=S^\ell_h\cap \widetilde{\Delta}_\ell, \quad 
P^{\ell,k,j}_h=S^\ell_h\cap (\omega_k\times H_j)\cap \widetilde{\Delta}_\ell 
\cap ([a,b] \times \Bbb R). 
\end{gather*} 
Then, arguing similarly to the case $(B.1)$, we can reach the conclusion of 
the theorem in the case $(B.3)$.   
\par 
Proof of Theorem \ref{thm1.2} for the case $(B.4)$. 
Apply the case $(B.3)$ with $\widetilde{\psi}(\xi_1)=\psi(-\xi_1)$ 
in place of $\psi$. Then by change of variables: $\xi_1\to -\xi_1$ 
we have the desired result.   
\par 
This completes the proof of Theorem  \ref{thm1.2} in its full generality.  
\par 
Theorem \ref{thm1.3} follows from Theorem \ref{thm1.2} 
in the same way as the theorem was shown in Section \ref{sec3} 
for the case $(B.1)$ with $\Psi'>0$. 
\par 
Now we prove Theorem \ref{thm1.1}.  
\par 
Proof of Theorem \ref{thm1.1} for the case  $(B.1)$.  
We have already proved the theorem in the case $(B.1)$ 
when $\Psi'>0$ on $I$.
Suppose that $\Psi'<0$ on $I$. We note that $0\notin I$, since if $0\in I$, 
then $\Psi'(0)=1/\psi(0)>0$. 
\par 
Recall that $I_1=[\sigma, \tau]$, $I_2=[\sigma', \tau']$, 
$I_3=[\sigma'', \tau'']$, 
$-b<\sigma''<\sigma'<\sigma<\tau<\tau'<\tau''<b$. 
We may assume that $\supp(\sigma_\lambda) \subset I_1\times \Bbb R$ and  
 $\Psi'<0$ on $I_3$,  $\Psi(t)=t/\psi(t)$.   
Let a function $\rho$ be defined on $C(a_1'',a_2'')$, 
$a_1''=\tau''/\psi(\tau'')$, $a_2''=\sigma''/\psi(\sigma'')$,   by 
\begin{equation*} 
\rho(u,v)=\frac{u}{\Psi^{-1}(\frac{u}{v})}, 
\end{equation*} 
where $\Psi: I_3 \to J:=\Psi(I_3)$, $\Psi^{-1}: J \to I_3$. 
We note that $\rho(u, \psi(u))=1$ for $u\in I_3$.  
Also, we may assume that 
$\supp(a)\subset C(a_1,a_2)\cap (I_1\times \Bbb R)$, $a_1=\tau/\psi(\tau)$, 
$a_2=\sigma/\psi(\sigma)$,  where $a$ is as in Theorem \ref{thm1.1}.  
\par 
 Then we have 
\begin{equation} \label{e5.3}
\widetilde{a}(\xi)(1-\rho(\xi))_+^\lambda
=a(\xi)(\xi_2-\psi(\xi_1))_+^\lambda  
\end{equation}
for some $\widetilde{a}\in C_0^\infty(\Bbb R^2)$ with  
$\supp(a)=\supp(\widetilde{a})$.   
To see this, note that $\partial_2 \rho<0$ on  $C(a_1'',a_2'')$. So, for 
$\xi\in \supp(a)$ we have   
\begin{align*} 
a(\xi)(1-\rho(\xi))_+^\lambda &=
a(\xi)(\rho(\xi_1,\psi(\xi_1))-\rho(\xi))_+^\lambda 
\\
&=a(\xi)\left(\int_0^1 \frac{\partial}{\partial t}
\rho(\xi_1,t(\psi(\xi_1)-\xi_2)+\xi_2)\, dt\right)_+^\lambda  \notag 
\\ 
&= a(\xi)\left((\psi(\xi_1)-\xi_2)\int_0^1 \partial_2 
\rho(\xi_1,t(\psi(\xi_1)-\xi_2)+\xi_2)\, dt\right)_+^\lambda  \notag
\\ 
&= a(\xi)(\xi_2-\psi(\xi_1))_+^\lambda
\left(\int_0^1 (-1)\partial_2 
\rho(\xi_1,t(\psi(\xi_1)-\xi_2)+\xi_2)\, dt\right)^\lambda  \notag
\\ 
&= a(\xi)h(\xi)(\xi_2-\psi(\xi_1))_+^\lambda     \notag
\end{align*} 
for some $h\in C_0^\infty(\Bbb R^2)$ such that $h \neq 0$ on $\supp(a)$. 
This implies \eqref{e5.3} by setting $\widetilde{a}=a/h$.     
\par 
Thus to prove Theorem \ref{thm1.1}  it suffices to show that  
$\|U_*^\lambda(f)\|_4 \leq C\|f\|_4$, where 
\begin{equation*} 
U_*^\lambda(f)(x)=\sup_{R>0}\left| 
\int \widetilde{a}(R^{-1}\xi)(1-\rho(R^{-1}\xi))_+^\lambda  
\hat{f}(\xi)e^{2\pi i\langle x,\xi\rangle}\, d\xi \right|.  
\end{equation*} 
We have $U_*^\lambda(f)
\leq CM(g_{\mathscr{F}^{-1}(\varphi^\delta)}(f))$ 
with some $\delta>-1/2$ by Corollary \ref{cor4.2}, where
$\varphi^\delta(\xi)=A(\xi)(1-\rho(\xi))^\delta_+$ and $A$ is related to 
$\widetilde{a}$ as $\widetilde{b}$ is related to $b$  
in Corollary \ref{cor4.2} with $\supp (A) \subset C(a_1,a_2)$.   
Thus Theorem \ref{thm1.1} follows from the estimates 
 $\|g_{\mathscr{F}^{-1}(\varphi^\delta)}(f)\|_4 
\leq C\|f\|_4$ for $\delta>-1/2$. To prove this we may 
replace $A(\xi)$ with $A_0(\xi)=w(\xi)A(\xi)$ for a suitable 
$w\in C_0^\infty(\Bbb R^2)$,  
since we may assume that 
\begin{equation*} 
\{\rho(\xi)<s, \, \xi\in C(a_1,a_2)\}
\subset \{|\xi|<Cs\}, \quad \forall s>0, 
\end{equation*}   
with some positive constant $C$. 
\par 
Choose $b_1(\xi_1)\in C_0^\infty(\Bbb R)$ such that $b_1=1$ on $I_2$ and 
$\supp(b_1)\subset I_3$ and let $A_0(\xi)=
b_1(\xi_1)A_0(\xi)+b_2(\xi_1)A_0(\xi)=: 
A_1(\xi)+A_2(\xi)$ with $b_2=1-b_1$.
Let $B_i(\xi):=A_i(\xi)(1-\rho(\xi))_+^\delta$, $i=1, 2$.  Then 
$B_2 \in C_0^\infty(\Bbb R^2)$ and 
$\supp(B_2)\subset \Bbb R^2\setminus \{0\}$. 
Thus $g_{\mathscr{F}^{-1}(B_2)}$ is bounded on $L^p$, $1<p<\infty$. 
On the other hand, since $\supp (A_1) \subset C(a_1,a_2)\cap(I_3\times \Bbb R)$,  arguing in the same way as in the proof of \eqref{e5.3},  we see that 
\begin{equation*} \label{e4.2}
B_1(\xi)
=\widetilde{A_1}(\xi)(\xi_2-\psi(\xi_1))_+^\delta, 
\end{equation*}
where $\widetilde{A_1}$ is in $C_0^\infty(\Bbb R^2)$ and 
$\supp(A_1)=\supp(\widetilde{A}_1)$.  
Therefore, by Theorem \ref{thm1.3} 
for the case $\lambda>-1/2$ and $\theta=0$, 
the boundedness of $g_{\mathscr{F}^{-1}(B_1)}$ follows. This completes   
the proof of Theorem \ref{thm1.1} in the case $(B.1)$ with $\Psi'<0$. 
\par 
Proof of Theorem \ref{thm1.1} in the case $(B.2)$. 
We have the following. 
\begin{proposition} \label{prop5.3} 
Let $\psi$, $b$, $\varphi_{\lambda,\theta}$ be as in Theorem $\ref{thm1.3}$. 
Put  
$\Phi_{\lambda,\theta}(\xi)
=b(\xi)\varphi_{\lambda,\theta}(\psi(\xi_1)-\xi_2)$. 
Then, under the same conditions on $\lambda$, $\theta$ as in Theorem 
$\ref{thm1.3}$ we have 
\begin{equation*} 
\|g_{\mathscr{F}^{-1}(\Phi_{\lambda,\theta})}(f)\|_4 \leq C\|f\|_4.   
\end{equation*} 
\end{proposition} 
This can be shown by using Theorem \ref{thm1.2} in the same way as 
Theorem \ref{thm1.3} is proved. 
Arguing as in the proof of Theorem \ref{thm1.1} for 
the case $(B.1)$ and using Proposition \ref{prop5.3}  
for the case $\lambda>-1/2$ and $\theta=0$, 
we can prove the following.

\begin{proposition} \label{prop5.4} 
In the case $(B.1)$ it holds that  
\begin{equation*} 
\|\widetilde{S}^\lambda_*f\|_4\leq C_\lambda\|f\|_4
\end{equation*} 
for $\lambda>0$, where 
\begin{equation*} 
\widetilde{S}^\lambda_*f(x)=\sup_{R>0}\left|\widetilde{S}^\lambda_Rf(x)\right|, \quad 
\widetilde{S}^\lambda_Rf(x)=\int_{\Bbb R^2} 
\widetilde{\sigma}_\lambda(R^{-1}\xi) 
\hat{f}(\xi)e^{2\pi i\langle x, \xi\rangle}\, d\xi,    
\end{equation*} 
 $\widetilde{\sigma}_\lambda(\xi)=
a(\xi)(\psi(\xi_1)-\xi_2)_+^\lambda$
 and $a$ is as in  Theorem $\ref{thm1.1}$. 
\end{proposition} 
By applying Proposition \ref{prop5.4} to $-\psi$ and by changing variables 
$\xi_2\to -\xi_2$ we can prove Theorem \ref{thm1.1} in the case $(B.2)$.  
\par 
Proof of Theorem \ref{thm1.1} in the case $(B.3)$.   
The proof is similar to that for the case $(B.1)$. In the proof we apply 
Theorem \ref{thm1.3} and the homogeneous function $\rho$ needed in applying 
Corollaries \ref{cor4.2}, \ref{cor4.4} 
is defined by using $\Psi_*(t)=\psi(t)/t$; it is of the form 
\begin{equation*} 
\rho(u,v)=\frac{u}{\Psi_*^{-1}\left(\frac{v}{u}\right)}.   
\end{equation*} 
\par 
Proof of Theorem \ref{thm1.1} in the case $(B.4)$.   
We apply results of the case $(B.3)$ to $\widetilde{\psi}(\xi_1)=\psi(-\xi_1)$ 
and apply the change of variables: $\xi_1 \to -\xi_1$.

\section{Proof of Theorem $\ref{thm1.5+}$}\label{sec6}  

It is sufficient to treat separately the cases $(B.i)$, $1\leq i\leq 4$. 
In each case we argue similarly as follows. 
\par 
By the Plancherel theorem we see that 
\begin{equation}\label{e.6.1}  
\|g_\eta(f)\|_2^2=\int_{\Bbb R^2}\int_0^\infty |\phi(t\xi)|^2\, \frac{dt}{t} 
|\hat{f}(\xi)|^2\, d\xi.  
\end{equation} 
Fix $\xi \in \Bbb R^2$.  Suppose that $t_0\xi\in \supp(\phi)$. Then by 
Lemma \ref{lem3.4} there exists $s_0\in [1-B_1\delta, 1+B_1\delta]$ such that 
$t_0\xi\in s_0\Gamma$.   Thus $\xi\in t_0^{-1}s_0\Gamma$.  If a positive 
number $t$ satisfies that $t\xi \in \supp(\phi)$, Lemma \ref{lem3.4} implies 
that $t\xi\in s\Gamma$ for some $s \in [1-B_1\delta, 1+B_1\delta]$. 
Since we have also $t\xi\in tt_0^{-1}s_0\Gamma$, it follow that  
$tt_0^{-1}s_0=s$ by Lemma \ref{lem3.5}. Thus we see that 
$tt_0^{-1}s_0 \in [1-B_1\delta, 1+B_1\delta]$, which implies that 
$t\in t_0s_0^{-1}[1-B_1\delta, 1+B_1\delta]$.  
Therefore we see that 
\begin{align*}
\int_0^\infty |\phi(t\xi)|^2\, \frac{dt}{t}&\leq 
\int_{t_0s_0^{-1}(1-B_1\delta)}^{t_0s_0^{-1}(1+B_1\delta)} |\phi(t\xi)|^2\, 
\frac{dt}{t}
\leq \|b\|_\infty^2\|\Phi\|_\infty^2\log\frac{1+B_1\delta}{1-B_1\delta} 
\\ 
&\leq C\|b\|_\infty^2\|\Phi\|_\infty^2\delta.  
\end{align*} 
Thus 
\begin{equation}\label{e.6.2} 
 \int_0^\infty |\phi(t\xi)|^2\, \frac{dt}{t}\leq  
C\|b\|_\infty^2\|\Phi\|_\infty^2\delta.
\end{equation} 
This is also true when there is no $t_0$ such that  $t_0\xi\in \supp(\phi)$. 
Thus \eqref{e.6.2} holds for all $\xi\in \Bbb R^2$. 
Using \eqref{e.6.2} in \eqref{e.6.1} and applying the Plancherel theorem 
again, we can obtain the conclusion of Theorem $\ref{thm1.5+}$.  
\par 
To conclude this note, 
we recall some results from \cite{Sa}. 
Let $\sigma_\lambda$, $\lambda>0$, and $S^\lambda_{R}f$ be as in  
\eqref{e1.1} and \eqref{e1.2}, respectively, 
where $\psi''$ is allowed to have a finite number of zeros of finite order in 
$I$.  The conditions $(A.1)$, $(A,2)$ need not be assumed. 
Then  the following vector valued inequality was proved in \cite{Sa}. 

\begin{theorem}\label{thm5.5} 
Let $\{R_\ell\}_{\ell=1}^\infty$ be a sequence of positive numbers and let   
 $p\in [4/3, 4]$.  Then  we have 
\begin{equation*}  
\left\|\left(\sum_{\ell=1}^\infty 
\left|S^\lambda_{R_\ell}f_\ell\right|^2\right)^{1/2}\right\|_p \leq 
C_\lambda \left\|\left(\sum_{\ell=1}^\infty |f_\ell|^2\right)^{1/2}\right\|_p.  
\end{equation*}  
\end{theorem} 
Using a special case of this, 
we can prove the following lacunary maximal theorem 
(see \cite[Remark 6.2]{Sa}).  
\begin{corollary} \label{cor5.6}
Let $\{R_\ell\}_{\ell=-\infty}^\infty$ be a sequence of positive numbers 
such that $1<q\leq \inf_\ell R_{\ell+1}/R_\ell$.  
Suppose that $S^\lambda_{R}f$ and $\sigma_\lambda$ are as in Theorem 
$\ref{thm5.5}$ and that the condition $(A.1)$ holds. 
Then we have 
\begin{equation*} \label{e.1.4}  
\left\|\sup_\ell\left|S^\lambda_{R_\ell}f\right|\right\|_p 
\leq C_\lambda\|f\|_p, \quad 4/3\leq p\leq 4.  
\end{equation*}
\end{corollary} 
Theorem \ref{thm5.5} was shown in \cite{Sa} from the results for the case when 
$\psi''\neq 0$ by an idea of H\"{o}rmander in \cite{Ho} and it was applied to 
show Corollary \ref{cor5.6} (see \cite{Sjo} for related results). 
Unfortunately, we cannot apply the idea to prove an analogue of 
Theorem \ref{thm1.1}  for $S^\lambda_{R}f$ with $\sigma_\lambda$ defined 
under the conditions of Corollary \ref{cor5.6}.

\end{document}